\def\draftdate{\today}

\documentclass{amsart}
\usepackage{amssymb}
\usepackage{xy}
\usepackage{accents}
\usepackage{bm}
\usepackage{color}

\newcommand{\cR}{\underline{\uR}}
\newcommand{\uR}{\underline{R}}
\newcommand{\cf}{\underline{\uf}}
\newcommand{\uf}{\underline{f}}
\newcommand{\cA}{\underline{\uA}}
\newcommand{\uA}{\underline{A}}
\newcommand{\cB}{\underline{\uB}}
\newcommand{\uB}{\underline{B}}
\newcommand{\cM}{\underline{\uM}}
\newcommand{\uM}{\underline{M}}

\newcommand{\uN}{\underline{N}}
\newcommand{\cT}{\underline{\uT}}
\newcommand{\uT}{\underline{T}}
\newcommand{\cX}{\underline{\uX}}
\newcommand{\uX}{\underline{X}}
\newcommand{\cY}{\underline{\uY}}
\newcommand{\uY}{\underline{Y}}
\newcommand{\cZ}{\underline{\uZ}}
\newcommand{\uZ}{\underline{Z}}

\newcommand{\uFp}{\underline{\bF}{}_{p}}
\newcommand{\cMUP}{\underline{\uMUP}}
\newcommand{\uMUP}{\underline{MUP}}
\newcommand{\cMUC}{\underline{\uMUC}}
\newcommand{\uMUC}{\underline{mu}}
\newcommand{\uBU}{\underline{BU}}
\newcommand{\uBUZ}{\Omega^{\infty}\underline{KU}}
\newcommand{\uBUC}{\Omega^{\infty}\underline{bu}}
\newcommand{\uMU}{\underline{MU}}
\newcommand{\uRep}{\underline{RU}}

\newcommand{\upi}{\underline{\pi}}
\newcommand{\uKU}{\underline{KU}}
\newcommand{\uG}{\underline{G}^{\op}}
\newcommand{\ubZ}{\underline{\bZ}^{\op}}
\newcommand{\uOZN}{(\Omega^{\infty}H\underline{\bZ}^{\op})_{\geq 0}}

\newcommand{\LTHH}{\bL THH}

\newcommand{\ssdot}{\bullet}
\newcommand{\supdot}{^\ssdot}
\newcommand{\subdot}{_\ssdot}

\newcommand{\overto}[1]{\xrightarrow{\,#1\,}}
\newcommand{\overfrom}[1]{\xleftarrow{\,#1\,}}

\let\Lim\lim
\def\lim{\Lim\nolimits}

\newcommand{\Spectra}{\aS}
\newcommand{\EQC}{\aE_{\Phi}}
\newcommand{\nEQC}{\aE}
\newcommand{\PreCyc}{\mathrm{Cyc}}

\mathchardef\varDelta="7101

\newcommand{\phat}[1][{p}]{^{\wedge}_{#1}}

\let\iso\cong
\let\sma\wedge

\renewcommand{\to}{\mathchoice{\longrightarrow}{\rightarrow}{\rightarrow}{\rightarrow}}
\newcommand{\from}{\mathchoice{\longleftarrow}{\leftarrow}{\leftarrow}{\leftarrow}}

\DeclareMathAlphabet{\catsymbfont}{U}{rsfs}{m}{n}
\newcommand{\aA}{{\catsymbfont{A}}}
\newcommand{\aAlg}{{\aA lg}}
\newcommand{\aC}{{\catsymbfont{C}}}
\newcommand{\aCom}{{\aC om}}
\newcommand{\aE}{{\catsymbfont{E}}}
\newcommand{\aF}{{\catsymbfont{F}}}

\newcommand{\aM}{{\catsymbfont{M}}}
\newcommand{\aMod}{{\aM\!od}}
\newcommand{\aP}{{\catsymbfont{P}}}
\newcommand{\aPair}{{\aP air}}
\newcommand{\aBPair}{{\aP air}_{\mathrm{pb}}}
\newcommand{\aS}{{\catsymbfont{S}}}

\newcommand{\aV}{{\catsymbfont{V}}}

\newcommand{\bC}{{\mathbb{C}}}

\newcommand{\bF}{{\mathbb{F}}}

\newcommand{\bL}{{\mathbb{L}}}
\newcommand{\bN}{{\mathbb{N}}}
\newcommand{\bS}{{\mathbb{S}}}

\newcommand{\bR}{{\mathbb{R}}}
\newcommand{\bT}{{\mathbb{T}}}
\newcommand{\bZ}{{\mathbb{Z}}}

\def\quickop#1{\expandafter\DeclareMathOperator\csname
#1\endcsname{#1}}
\quickop{id}\quickop{Id}\quickop{holim}\quickop{hocolim}\quickop{op}
\quickop{co}\quickop{Ar}\quickop{sing}\quickop{Hom}\quickop{w}\quickop{Ho}
\quickop{ob}\quickop{diag}\quickop{Cyc}\quickop{Fib}\quickop{Cof}
\quickop{tr}\quickop{trc}\quickop{Gal}\quickop{colim}\quickop{triv}
\quickop{red}\quickop{hoeq}\quickop{inc}\quickop{cyc}
\quickop{fin}\quickop{cy}\quickop{Fix}\quickop{Tot}\quickop{res}\quickop{Sp}
\quickop{Span}\quickop{Gpd}\quickop{Cat}\quickop{Fin}\quickop{Mod}\quickop{CAlg}
\quickop{Vect}\quickop{Fun}\quickop{can}\quickop{HoEq}

\DeclareMathOperator*\lcolim{colim}

\newtheorem{main}{Theorem}

\numberwithin{equation}{section}
\newtheorem{thm}[equation]{Theorem}
\newtheorem*{thm*}{Theorem}

\newtheorem{prop}[equation]{Proposition}

\theoremstyle{definition}
\newtheorem{defn}[equation]{Definition}
\newtheorem{notn}[equation]{Notation}
\newtheorem{ter}[equation]{Terminology}
\newtheorem{cons}[equation]{Construction}

\newtheorem{rem}[equation]{Remark}
\newtheorem{example}[equation]{Example}
\newtheorem{warning}[equation]{Warning}

\xyoption{arrow}
\xyoption{curve}
\xyoption{matrix}
\xyoption{cmtip}
\SelectTips{cm}{}

\newcommand{\term}[1]{\textit{#1}}

\newdir{ >}{{}*!/-5pt/\dir{>}}

\begin{document}

\title[Relative cyclotomic structures]%
{Relative cyclotomic structures and equivariant complex cobordism}

\author{Andrew J. Blumberg}
\address{Department of Mathematics, Columbia University, 
New York, NY \ 10027}
\email{blumberg@math.columbia.edu}
\thanks{The first author was supported in part by NSF grants
DMS-2104420, DMS-2405029}
\author{Michael A. Mandell}
\address{Department of Mathematics, Indiana University,
Bloomington, IN \ 47405}
\email{mmandell@iu.edu}
\thanks{The second author was supported in part by NSF grants
DMS-2104348, DMS-2405030}
\author{Allen Yuan}
\address{Department of Mathematics, Northwestern University,
Evanston, IL \ 60208}
\email{allenyuan@northwestern.edu}
\thanks{The third author was supported in part by NSF grant DMS-2002029}

\date{\draftdate}

\begin{abstract}
We describe a structure on a commutative ring (pre)cyclotomic
spectrum $\cR$ that gives rise to a (pre)cyclotomic structure on
topological Hochschild homology ($THH$) relative to its underlying
commutative ring spectrum.
This lets us construct $TC$ relative to $\cR$, denoted $TC^{\cR}$, and
we prove some descent results relating $TC^{\cR}$ and $TC$.
We explore several examples of this structure on familiar
$\bT$-equivariant commutative ring spectra including the periodic $\mathbb{T}$-equivariant
complex cobordism spectrum $MUP_{\mathbb{T}}$ and a new (connective)
equivariant version of the complex cobordism spectrum $MU$.
\end{abstract}

\maketitle

\section{Introduction}

The remarkable success of trace methods over the past 30 years derives
from the relationship between algebraic $K$-theory and topological
cyclic homology ($TC$).  In contrast to algebraic constructions of
cyclic homology, $TC$ intrinsically depends on topology, namely the
cyclotomic structure on topological Hochschild homology ($THH$).  The
cyclotomic structure ultimately depends on working over the sphere
spectrum $\bS$; this structure does not exist in algebra, working over
the integers $\bZ$.  Moreover, the surprising use of $THH$ and $TC$ in
recent seminal work of Bhatt-Morrow-Scholze~\cite{BMS, BMS2} on
$p$-adic Hodge theory reveals another fundamental role of the
cyclotomic structure in modern geometry.

The construction of $THH$ makes sense relative to any commutative ring
orthogonal spectrum $R$, but the cyclotomic structure does not.  The
search for cyclotomic structures on relative $THH$ goes back to the
1990s.  From the perspective of trace methods, relative cyclotomic
constructions were supposed to give descent spectral sequences
computing the absolute construction for further algebraic $K$-theory
computations.  More recently,~\cite{BMS2} constructed cyclotomic
structures relative to the group ring $\bS[t]$ and used them to build
Breuil-Kisin modules.  New work in Floer homotopy theory shows that
spectral Fukaya categories will often come with enrichments over some
form of complex cobordism; relative cyclotomic structures here are
hoped to correspond to interesting and effective structure on
symplectic cohomology.

Previous work of the first and second author and collaborators
\cite[\S7]{ABGHLM} studies the problem of constructing cyclotomic
structures on relative $THH$.  Recall from~\cite[4.1]{BM-cycl} that a
$p$-precyclotomic structure on a $\bT$-equivariant orthogonal spectrum
$\uX$ consists of a $\bT$-equivariant 
map
\[
\rho^{*}\Phi^{C_{p}}\uX\to \uX
\]
where $\rho^{*}$ is change of groups along the $p$th root isomorphism
$\rho \colon \bT\overto{\iso} \bT/C_{p}$ and $\Phi^{C_{p}}$ denotes
the point-set multiplicative geometric fixed point functor
of~\cite[V.4.3]{MM}.  The $p$-precyclotomic spectrum is $p$-cyclotomic
if the induced map from the derived geometric fixed points to $\uX$ is
an $\aF_{p}$-equivalence where $\aF_{p}$ is the family of $p$-groups
$C_{p^{n}}<\bT$.  (We review these and related structures, including
associative and commutative ring $p$-(pre)cyclotomic structures, in
Section~\ref{sec:cyclotomic}.)  In this paper, we typically denote a
$p$-precyclotomic spectrum with a double underlined symbol, e.g.,
$\cR$, and use the notational shorthand of the single underlined
symbol $\uR$ to denote its underlying equivariant spectrum (for the
group $\bT$ or sometimes $C_{p}$) and the non-underlined symbol $R$
for the underlying non-equivariant spectrum.

Cyclotomic structures have been the main focus in the literature on
trace methods because (absolute) $THH$ naturally can be endowed with
such a structure.  However,  $p$-precyclotomic spectra provide
a minimal structure sufficient for constructing ($p$-typical) $TC$ (e.g.,
see~\cite[6.7]{BM-cycl}).  We focus on precyclotomic spectra in this
paper because most of the structures we describe on examples of
interest are not $p$-cyclotomic but only $p$-precyclotomic. (See
Section~\ref{sec:phiinfty} for a discussion on converting
$p$-precyclotomic to $p$-cyclotomic spectra with equivalent $TC$.)

Given a commutative ring
$p$-(pre)cyclotomic spectrum $\cR$, we then have a canonical
counit map $THH(R)\to \uR$ of $\bT$-equivariant commutative ring
orthogonal spectra. 
If this map is $p$-(pre)cyclotomic, then the $R$-relative topological
Hochschild homology $THH^{R}(-)$ has a natural $p$-(pre)cyclotomic
structure.  This condition is not formal (see \cite[p.~2146]{ABGHLM})
and it is not a priori clear when to expect it to hold.

The purpose of this paper is to give a new framework for constructing
$p$-pre\-cyclotomic structures on relative $THH$, which implies the criterion
of \cite[7.6]{ABGHLM} and which we can check in some interesting
examples.  Our framework depends on a new self map associated to a
commutative ring $p$-precyclotomic spectrum $\cR$.  In
Section~\ref{sec:power}, we construct a self map
$\Psi \colon R \to R$ in the homotopy category of non-equivariant commutative
ring orthogonal spectra.  The map $\Psi$ is the composite of the
multiplicative transfer
\[
R \to \Phi^{C_p} \uR
\]
(see also Definition~\ref{def:power-operation})
and the non-equivariant map underlying the $p$-precyclotomic structure map
\[
\rho^{*}\Phi^{C_{p}}\uR\to \uR.
\]
We refer to $\Psi$ as the $p$-cyclotomic power operation and prove the
following theorem, which we state jointly in the $p$-cyclotomic and
$p$-precyclotomic cases.

\begin{main}\label{main:crit}
Let $\cR$ be a commutative ring $p$-(pre)cyclotomic spectrum, let $R$
denote its underlying non-equivariant commutative ring orthogonal spectrum, and
assume that $\Psi \colon R\to R$ is the identity in the homotopy
category of commutative ring orthogonal spectra.
Then $\cR$-relative topological Hochschild homology $THH^{\cR}(-)$ can
be given the structure of a $p$-(pre)cyclotomic $\cR$-module.
\end{main}

We then get $TC$ relative to $\cR$, $TC^{\cR}$, by applying to 
$THH^{\cR}$ the classic $p$-typical topological cyclic homology
construction $TC(-;p)$ (see for example~\cite[6.3]{BM-cycl}).  For
$\cR=\bS$ with its canonical cyclotomic structure, $TC^{\bS}(A)$ is the
usual ($p$-typical) $TC(A)$.

In the case when $\cR=THH(A)$ for some commutative ring orthogonal
spectrum $A$,
the $p$-cyclotomic power operation is the Adams operation $\psi^{p}$
of~\cite[\S10]{ABGHLM}, which is typically not the identity for
$A\not\simeq \bS$. (See Proposition~\ref{prop:PsiTHH}.)

We call a commutative ring $p$-precyclotomic spectrum along with a
choice of homotopy from $\Psi$ to the identity a
\term{$p$-(pre)cyclotomic base}.  Not all $\bT$-equivariant
commutative ring spectra admit such a structure: the standard
equivariant structure on complex $K$-theory cannot be a
$p$-(pre)cyclotomic base (see Example~\ref{ex:Kth}); however, in
Section~\ref{sec:examples}, we show that a number of interesting
spectra do. Notably, we show that $\bS[t]$, $\bF_p$, and $MUP$ (with
given $p$-(pre)cyclotomic structures explained there) all admit the
structure of a $p$-(pre)cyclotomic base. In addition, we construct
such a structure on $MU$ for a new precyclotomic structure, which we
denote $\cMUC$. (Note: the underlying $\bT$-equivariant spectrum
$\uMUC$ of
$\cMUC$ is not homotopical or geometric $\bT$-equivariant complex
cobordism, but is a new equivariant connective Thom spectrum with
underlying non-equivariant Thom spectrum $MU$.)

\begin{main}\label{main:MUP}
The genuine $\bT$-equivariant commutative ring orthogonal spectrum
$\uMUP=MUP_{\bT}$ admits the structure of a
$p$-precyclotomic base.  The commutative ring spectrum $MU$ admits the
structure of a $p$-precyclotomic base (with a new $\bT$-equivariant
structure). 
\end{main}

Finally, we state a descent theorem (proved as
Theorem~\ref{thm:descent} in Section~\ref{sec:descent}) relating
relative $TC$ to absolute $TC$ for a connective $p$-precyclotomic base
$\cR$ and connective $R$-algebra $A$.  In the statement, we use the
commutative $R$-algebra structure on $A$ to get a commutative
$R^{(n+1)}$-algebra structure on $A$ (for all $n\geq -1$) by
restriction along the iterated multiplication $R^{(n+1)}\to R$.  We
argue in Section~\ref{sec:descent} that the Adams resolution of $\bS$
with respect to $R$, $R^{(\bullet+1)}$, has the canonical structure of 
a cosimplicial object in the category of $p$-precyclotomic bases,
$\cR^{(\bullet+1)}$, and using $TC^{(-)}$ as a functor of the
$p$-precyclotomic base, we get an augmented cosimplicial spectrum 
\[
TC^{\cR^{(\bullet+1)}}(A)
\]
(where for $\bullet=-1$, $\cR^{(\bullet+1)}=\bS$).

\begin{main}\label{main:descent}
Let $\cR$ be an $\aF_{p}$-connective $p$-precyclotomic base, whose
underlying commutative ring $p$-precyclotomic spectrum is cofibrant. Let
$A$ be a connective cofibrant commutative $R$-algebra.
The canonical map
\[
TC(A)\to \Tot( TC^{\cR^{(\bullet+1)}}(A))
\]
is a weak equivalence.
\end{main}

\subsection*{Acknowledgments}
The authors would like to thank Jeremy Hahn and Inbar Klang for useful
discussions.  They thank Mohammed Abouzaid for insisting that $MUP$
should have some kind of cyclotomic structure and helpful
conversations.  This work owes a lot to the previous collaboration of
the first two authors with Vigleik Angeltveit, Teena Gerhardt, Mike
Hill, and Tyler Lawson on cyclotomic structures via the norm; we thank
them for years of productive conversations.

\section{Equivariant preliminaries}\label{sec:equivariant}

In this section, we review the prerequisites in equivariant stable
homotopy theory needed for defining and working with precyclotomic
and cyclotomic spectra.  Although modern treatments of cyclotomic
spectra minimize the explicit use of genuine equivariant stable
homotopy theory in the setup and require none for end-users, no such
simplification can exist for precyclotomic spectra, especially in the
nonconnective setting.  Our main examples of interest, the 
cobordism spectra $\cMUP$ and $\cMUC$, have precyclotomic structures
that are not cyclotomic structures.

We take as our model for the equivariant stable category the point-set
category of $\bT$-equivariant orthogonal spectra indexed on the complete
universe 
\[
U=\bigoplus_{n\in \bZ}\bC(n)^{\infty}
\]
where $\bC(0)$ denotes the complex numbers with the trivial $\bT$-action,
$\bC(1)$ denotes the complex numbers with its natural $\bT$-action (as
the unit complex numbers), and in general, for any $n$, the element
$\zeta \in \bT$ acts on $\bC(n)$ as multiplication by $\zeta^{n}$.
For such a
spectrum $T$, homotopy groups are defined by 
\[
\pi^{H}_{q}T=\lcolim_{V<U}\lcolim_{n\geq -q} \pi_{n+q}((\Omega^{V}T(V\oplus \bR^{n}))^{H})
\]
for $H<\bT$ a closed subgroup, with the colimit over the finite
dimensional $\bT$-stable real vector subspaces of $U$. We work with
the family $\aF_{\fin}$ 
of finite subgroups of $\bT$ and the family $\aF_{p}\subset \aF_{\fin}$
of $p$-subgroups of $\bT$.  For any family $\aF$, an
\term{$\aF$-equivalence} is a map that induces
an isomorphism on $\pi_{*}^{H}$ for $H\in \aF$
(but not necessarily for $H\notin\aF$).  The $\aF$-equivalences are
the weak equivalences in the $\aF$-model
structures in Theorem~C 
of~\cite{BM-cycl23}, which are ``convenient'' (in the sense
of~\cite{Shipley-Convenient}) in that for these model structures,
cofibrant $\bT$-equivariant commutative ring orthogonal spectra have
their underlying $\bT$-equivariant ring orthogonal spectra cofibrant. 
They also have the usual free/forgetful Quillen adjunctions:
the forgetful functor from  
$\bT$-equivariant associative ring and commutative ring orthogonal spectra to
$\bT$-equivariant orthogonal spectra creates the weak equivalences and
fibrations for the $\aF$-model structures on the categories of $\bT$-equivariant
associative and commutative ring orthogonal spectra.

We will often take a $\bT$-equivariant orthogonal spectrum (indexed on
$U$) and look at its underlying non-equivariant spectrum indexed on 
$\bR^{\infty}$.  As indicated in the introduction, we will typically
denote an equivariant object with an underlined symbol (e.g., $\uR$)
and its underlying non-equivariant object with the non-underlined
symbol ($R$ in the case of $\uR$).  When we need to explicitly denote the forgetful
functor, we use the following notation.

\begin{notn}\label{notn:forget}
Let $i$ denote the forgetful functor from $\bT$-equivariant orthogonal spectra (indexed on
$U$) to non-equivariant spectra (indexed on $\bR^{\infty}$), and its
structured variants for categories of commutative ring orthogonal
spectra, associative ring orthogonal spectra, and modules over
commutative or associative ring orthogonal spectra.
\end{notn}

We write $\Phi$ for the endofunctor on $\bT$-equivariant orthogonal
spectra obtained as the composite of the (point-set multiplicative)
$C_{p}$ geometric fixed point functor $\Phi^{C_{p}}$
of~\cite[V.4.3]{MM} followed by the change of group functor along the $p$th root isomorphism $\rho \colon
\bT\iso \bT/C_{p}$:

\begin{notn}\label{notn:Phi}
Let $\Phi \uX:=\rho^{*}(\Phi^{C_{p}}\uX)$.
\end{notn}

The change of group functor $\rho^{*}$ implicitly involves a change
of universe functor, using the isomorphism $\rho^{*}U^{C_{p}}\iso U$
that comes from the standard isomorphisms $\rho^{*}\bC(pn)\iso \bC(n)$.
The functor $\Phi$ comes with a lax symmetric monoidal structure: 
a map
\[
\iota\colon \bS\to \Phi \bS
\]
and a natural transformation 
\[
\lambda \colon \Phi \uX\sma \Phi \uY\to \Phi (\uX\sma \uY)
\]
that is coherently associative, commutative, and unital (in the
point-set category of $\bT$-equivariant orthogonal spectra). The map
$\iota$ is an isomorphism and $\lambda$ is often an isomorphism
according to~\cite[V.4.7]{MM}. The
paper~\cite{BM-cycl23} expands the conditions under which $\lambda$ is
an isomorphism: it
defines in~III.1.5 a large 
class $\EQC$ of $\bT$-equivariant orthogonal
spectra $\uX$ with the property that $\lambda$ is an isomorphism for
any $\uY$. (They also have several additional properties including
that the map from the derived functor to the point-set functor
$\bL\Phi \uX\to \Phi \uX$ is an isomorphism and that $\uX$ is flat for
the smash product of $\bT$-equivariant orthogonal spectra.)
The key result of \cite[III.1.6]{BM-cycl23} is that the
class $\EQC$  
contains the cofibrant $\bT$-equivariant orthogonal spectra (and
therefore the cofibrant $\bT$-equivariant associative ring orthogonal
spectra), the cofibrant $\bT$-equivariant commutative ring orthogonal
spectra, and the cofibrant $\bT$-equivariant $\uA$-modules for any
$\bT$-equivariant associative ring orthogonal spectrum whose
underlying $\bT$-equivariant orthogonal spectrum is in $\EQC$.  (It is
expected that $\lambda$ may not be an isomorphism in general, but no
counterexample is currently known to the authors.)

We also have a non-equivariant version of the class $\EQC$ that
consists of the cofibrant objects in the ``convenient $\Sigma$-model
structure'' of~\cite[\S III.4]{BM-cycl23}
(essentially due to \cite[3.2]{Shipley-Convenient}); we denote this
class $\nEQC$.  When $A$ is an associative ring
orthogonal spectrum whose underlying orthogonal spectrum is in the
non-equivariant class $\nEQC$, then $THH(A)$ is in the equivariant
class $\EQC$.  This happens in particular when $A$ is a cofibrant
$R$-algebra for $R$ a commutative ring orthogonal spectrum that is
cofibrant in the standard model structure of~\cite[15.1]{MMSS} or the
model structure 
of~\cite[III.4.1]{BM-cycl23}.  Moreover, when the
underlying orthogonal spectrum of $A$ is in the class $\nEQC$, the
diagonal map $THH(A)\to \Phi THH(A)$ of~\cite[2.19]{ABGHLM} is an 
isomorphism (see \cite[VI.7.9]{BM-cycl23} and the proof
of Theorem~4.7 in~\cite[p.~2134]{ABGHLM}); this endows $THH(A)$ with a
cyclotomic structure as in~\cite[1.5]{ABGHLM}.

Using the lax symmetric monoidal structure, $\Phi$ refines to an
endofunctor on various categories of structured spectra, including
the category of $\bT$-equivariant associative ring orthogonal spectra and
the category of $\bT$-equivariant commutative ring orthogonal
spectra.  For a $\bT$-equivariant associative ring orthogonal spectrum
$A$, $\Phi$ refines to a functor from $\uA$-modules to $\Phi \uA$-modules.

The endofunctor $\Phi$ on $\bT$-equivariant orthogonal spectra has a
left derived functor, $\bL\Phi$, which can be computed by applying
$\Phi$ to a cofibrant approximation (in the standard or positive stable model
structures of~\cite[III.4.2,III.5.3]{MM}, the $\aF_{\fin}$- or
$\aF_{p}$-model structures of~\cite[IV.6.5]{MM}, or their variants
in~\cite[Thm.~C]{BM-cycl23}). This also works to
construct the derived functor
on the category of $\bT$-equivariant associative ring orthogonal
spectra.  For the category of $\bT$-equivariant commutative ring
orthogonal spectra, there is a standard $\aF_{p}$-model structure
constructed like the standard model structure of~\cite[III.8.1]{MM}
but building the cofibrations using only orbits from subgroups in
$\aF_{p}$, as in the non-multiplicative version in~\cite[IV.6.5]{MM}.
The functor $\Phi$ does not preserve weak
equivalences between objects cofibrant in the standard model structure
of~\cite[III.8.1]{MM} or the standard $\aF_{p}$-model
structure on $\bT$-equivariant
commutative ring orthogonal spectra, but (III.1.9)
of~\cite{BM-cycl23} displays a zigzag of Quillen equivalences induced
by the identity functor 
\[
\text{standard}\from
\text{model}(\aA^{\natural})\to
\text{model}(\aA)
\]
(left adjoints displayed)
between the standard $\aF_{p}$-model
structure on $\bT$-equivariant
commutative ring orthogonal spectra and two new model structures with
the property that $\Phi$ preserves weak equivalences between cofibrant
objects.  In the middle
is the $\aF_{p}$-local $\aV_{\pi!}$-restricted model
structure of Example~II.1.3 and Theorem~II.1.9
of~\cite{BM-cycl23} and on the right is the model structure of
Theorem~C of~\cite{BM-cycl23}.  The latter has the
further property that $\Phi$ preserves cofibrations and acyclic
cofibrations.  Moreover, Theorem~C
of~\cite{BM-cycl23} shows that the derived functor of $\Phi$ 
for $\bT$-equivariant commutative ring orthogonal spectra agrees with
the derived functor of $\Phi$ on the underlying $\bT$-equivariant
orthogonal spectra.

We write $N_{e}^{C_{p}}$ for the Hill-Hopkins-Ravenel
norm~\cite[\S 2.2.3]{HHR}, \cite[\S 2.2]{ABGHLM} from (non-equivariant)
orthogonal spectra to $C_{p}$-equivariant orthogonal spectra (indexed
on $U$ restricted to $C_{p}$). This functor is defined on all
orthogonal spectra but when restricted to commutative ring orthogonal
spectra it gives the free functor from the category of commutative
ring orthogonal spectra to the category of $C_{p}$-equivariant
commutative ring orthogonal spectra
\[
\aCom^{C_{p}}(N_{e}^{C_{p}}A,\uB)\iso \aCom(A,B)
\]
(for any commutative ring orthogonal spectrum $A$ and 
$C_{p}$-equivariant commutative ring orthogonal spectrum $\uB$).
As such, a $C_{p}$-equivariant
commutative ring orthogonal spectrum $\uR$ comes with a canonical map of
$C_{p}$-equivariant commutative ring orthogonal spectra 
\begin{equation}\label{eq:counit}
N_{e}^{C_{p}}R\to \uR
\end{equation}
given by the counit of the adjunction.

As explained in~\cite{ABGHLM}, $THH(-)$ as a functor from associative
ring orthogonal spectra to $\bT$-equivariant orthogonal spectra is the
norm $N_{e}^{\bT}(-)$.  On commutative ring orthogonal spectra, $THH$
is the free functor to $\bT$-equivariant commutative ring orthogonal
spectra. (This refines a 1997 theorem of
McClure-Schw\"anzl-Vogt~\cite{MSV}.) As a consequence, for $\uR$ a
$\bT$-equivariant 
commutative ring orthogonal spectrum and $A$ a 
commutative ring orthogonal spectrum, we have an adjunction
\[
\aCom^{\bT}(THH(A),\uR)\iso \aCom(A,R)
\]
(homeomorphism of mapping spaces) relating maps of commutative ring
orthogonal spectra to maps of $\bT$-equivariant commutative ring
orthogonal spectra.  The derived functor of $THH$ represents the
derived functor of the free functor, and we have the analogous
adjunction in the homotopy category.  We summarize the situation in
the following proposition.

\begin{prop}[{\cite[4.3]{ABGHLM}, \cite{MSV}}]\label{prop:tensor}
Restricted to the category of commutative ring orthogonal spectra,
$THH$ is the free functor from commutative ring orthogonal spectra to
$\bT$-equivariant commutative ring orthogonal spectra,
\[
\aCom^{\bT}(THH(-),-)\iso \aCom(-,i(-)).
\]
\end{prop}

\section{Relative $THH$}

The idea of $THH$ relative to a commutative ring orthogonal spectrum
$R$ is to use the smash 
product of $R$-modules $\sma_{R}$ in place of the smash product of
spectra in the cyclic bar construction.  In this section, we review
and clarify the $\bT$-equivariant structures on relative $THH$,
particularly in the case when $R$ has the extra structure of being
the underlying non-equivariant commutative ring orthogonal spectrum of a
$\bT$-equivariant commutative ring orthogonal spectrum $\uR$.

We start with $R$ a commutative ring orthogonal spectrum and $A$ an
associative $R$-algebra.  We then have a cyclic bar construction
$N^{\cy}_{R}$ constructed using the smash product over $R$,
$\sma_{R}$ of~\cite[22.2]{MMSS}:
\[
N^{\cy}_{R}(A)=| q \mapsto
\underbrace{A\sma_{R}\dotsb\sma_{R}A}_{q+1\text{\ factors}}|.
\]
This is the geometric realization of a cyclic spectrum, and that
structure endows it with
a natural $\bT$-action.  We do point-set change of universe
$I_{\bR^{\infty}}^{U}$ to get a $\bT$-equivariant orthogonal spectrum
indexed on $U$.  

\begin{defn}
For a commutative ring orthogonal spectrum $R$, define the point-set
functor $THH^{\epsilon R}$ from the category of associative $R$-algebras to the
category of $\bT$-equivariant orthogonal spectra as the composite of
the cyclic bar construction with respect to $\sma_{R}$ and change of universe:
\[
THH^{\epsilon R}(A):=I_{\bR^{\infty}}^{U}N^{\cy}_{R}(A).
\]
We understand $THH$ with no superscript as $THH^{\epsilon \bS}$; this
is the $THH$ functor from associative ring orthogonal spectra to
$\bT$-equivariant orthogonal spectra of~\cite{ABGHLM}.
\end{defn}

(The reason for the $\epsilon$ in the notation will become clear in
Proposition~\ref{prop:relissma} and Definition~\ref{defn:eqrelthh} below.)

The previous definition does not capture all the structure:
$N^{\cy}_{R}(A)$ has the structure of an $R$-module, but we have
defined $THH^{\epsilon R}(A)$ only as an equivariant spectrum.  To describe the
equivariant module structure, we use the following notation.  It
expresses a particular point-set construction of a $\bT$-equivariant spectrum
from a non-equivariant spectrum, viewed as an equivariant spectrum on
the naive universe $\bR^{\infty}$ by giving it the trivial action.

\begin{notn}\label{notn:epsilon}
For a non-equivariant orthogonal spectrum $X$, let $\epsilon
X=I_{\bR^{\infty}}^{U}X$ denote the $\bT$-equivariant orthogonal
spectrum (indexed on $U$) obtained from change of universe from $X$ viewed
as a $\bT$-equivariant orthogonal spectrum indexed on $\bR^{\infty}$
with the trivial $\bT$-action.
\end{notn}

We emphasize that in the previous notation, $\epsilon$ denotes a
point-set functor.  It preserves associative ring and commutative ring
structures. It also preserves weak equivalences between cofibrant
orthogonal spectra and between cofibrant associative ring orthogonal
spectra (even for maps of orthogonal spectra that are not 
ring maps).  It preserves ring map weak equivalences between cofibrant
commutative ring orthogonal spectra (in the standard model structures or
the model structure of~\cite[III.4.1]{BM-cycl23});
however, we repeat the warning of~\cite[6.1]{ABGHLM}. 

\begin{warning}\label{warn:epsilon}
The functor $\epsilon$ preserves weak equivalences between cofibrant
objects in the category of commutative ring orthogonal spectra, and so
admits a left derived functor $\epsilon^{\bL}$ from the homotopy category of
commutative ring orthogonal spectra to the homotopy category of
$\bT$-equivariant commutative ring orthogonal spectra.  However, the composite
with the derived forgetful functor to the equivariant stable category is not the
left derived functor of $\epsilon$ from orthogonal spectra to
$\bT$-equivariant orthogonal spectra.
\[
\xymatrix@-1pc{%
\Ho\aCom \ar[r]^{\epsilon^{\bL}}\ar[d]\ar@{}[dr]|{\color{red}{\displaystyle\bm\times}}
&\Ho\aCom^{\bT}\ar[d]\\
\Ho\Spectra\ar[r]_{\epsilon^{\bL}}&\Ho\Spectra^{\bT}
}
\]
Specifically, let $R$ be a
cofibrant commutative ring orthogonal spectrum and let $X\to R$ be a
cofibrant approximation in the category of orthogonal spectra.  Then
the underlying non-equivariant spectrum of the derived geometric fixed
point functor
$\bL\Phi^{C_{p}} \epsilon X$ is equivalent to $X$, whereas
by Theorem~A of~\cite{BM-cycl23}, the underlying non-equivariant 
spectrum of $\bL\Phi^{C_{p}} \epsilon R$ is equivalent to $R\otimes(*\amalg
BC_{p})$ (tensor in the category of commutative ring orthogonal spectra).
Because of this possible point of confusion, we \emph{never} again
explicitly refer to a derived functor for $\epsilon$ in this paper;
however, since we often ask for $R$ to be a cofibrant commutative ring
orthogonal spectrum, we are implicitly using the commutative ring
derived functor. 
\end{warning}

We lift $THH^{\epsilon R}$ as a functor to $\bT$-equivariant $\epsilon
R$-modules using the following result of~\cite[1.8]{ABGHLM}.
In it, the $THH(R)$-module structure on $\epsilon R$ comes from the
map of $\bT$-equivariant commutative ring orthogonal spectra
$THH(R)\to \epsilon R$  adjoint under Proposition~\ref{prop:tensor} to the
canonical isomorphism $R\iso i\epsilon R$.

\begin{prop}[{\cite[1.8]{ABGHLM}}]\label{prop:relissma}
Let $R$ be a commutative ring orthogonal spectrum and $A$ an
associative $R$-algebra.
There is a natural isomorphism of $\bT$-equivariant orthogonal spectra 
\[
THH^{\epsilon R}(A)\iso THH(A)\sma_{THH(R)}\epsilon R.
\]
Moreover, if $R$ is cofibrant as a commutative ring orthogonal
spectrum (in the standard model structure or the
model structure of~\cite[III.4.1]{BM-cycl23}), and $A$
is a cofibrant associative $R$-algebra or a cofibrant 
commutative $R$-algebra, then $THH^{\epsilon R}(A)$ represents the derived
smash product of $THH(R)$-modules
\[
THH^{\epsilon R}(A)\simeq THH(A)\sma^{\bL}_{THH(R)}\epsilon R.
\]
\end{prop}

The left derived functor $(-)\sma^{\bL}_{THH(R)}\epsilon R$ from
$\bT$-equivariant $THH(R)$-modules to $\bT$-equivariant orthogonal
spectra factors as the composite of the derived forgetful functor and
the derived functor $(-)\sma^{\bL}_{THH(R)}\epsilon R$ from
$\bT$-equivariant $THH(R)$-modules to $\bT$-equivariant $\epsilon
R$-modules.
\[
\xymatrix@R+1pc{%
\Ho\aMod^{\bT}_{THH(R)}\ar[rr]^{(-)\sma^{\bL}_{THH(R)}\epsilon R} \ar@/_2pc/[rrr]_{(-)\sma^{\bL}_{THH(R)}\epsilon R}
& \, &\Ho\aMod^{\bT}_{\epsilon R} \ar[r] & \Ho\Spectra^{\bT}
}
\]

The previous result also implies the following homotopical property of
$THH^{\epsilon R}$.

\begin{prop}
Let $R'\to R$ be a weak equivalence of cofibrant commutative ring
orthogonal spectra, let $A'$ be a cofibrant associative $R'$-algebra
and let $A$ be either a cofibrant associative $R$-algebra or a
cofibrant commutative $R$-algebra.  A weak equivalence of
$R'$-algebras $A'\to A$ then induces an $\aF_{\fin}$-equivalence of
$\bT$-equivariant $\epsilon R'$-modules $THH^{\epsilon R'}(A')\to
THH^{\epsilon R}(A)$.
\end{prop}

Now let $\uR$ be a $\bT$-equivariant commutative ring orthogonal
spectrum and let $R=i\uR$. We then get a map of $\bT$-equivariant
commutative ring orthogonal spectra $THH(R)\to \uR$ adjoint
under Proposition~\ref{prop:tensor} to the identity map $R\to i\uR$.  This endows
relative $THH$ with a different equivariant structure that
incorporates the equivariance on $\uR$.

\begin{defn}\label{defn:eqrelthh}
Let $\uR$ be a $\bT$-equivariant commutative ring orthogonal spectrum
and let $R=i\uR$ be the underlying non-equivariant commutative ring orthogonal spectrum.
Let 
\[
THH(R)\to \uR
\]
be the map of $\bT$-equivariant commutative ring orthogonal spectra
adjoint under Proposition~\ref{prop:tensor} to the identity map $R\to i\uR$. 
Define the functor $THH^{\uR}$ from the category of (non-equivariant)
$R$-algebras to the category of $\bT$-equivariant $\uR$-modules by
\[
THH^{\uR}(A):=THH(A)\sma_{THH(R)}\uR.
\]
\end{defn}

The same argument as Proposition~\ref{prop:relissma}
proves the following results.

\begin{prop}\label{prop:relisdersma}
Let $\uR$ be a cofibrant $\bT$-equivariant commutative ring orthogonal
spectrum (in the standard model structure or the model structure of
\cite[Thm.~C]{BM-cycl23}) and let $R=i\uR$ be the underlying
non-equivariant commutative ring orthogonal spectrum.  Let $A$ be a
cofibrant associative $R$-algebra or a cofibrant commutative
$R$-algebra (in the standard model structures or the model structures of
\cite[III.4.1]{BM-cycl23}).  Then $THH^{\uR}(A)$ represents the derived
smash product of $THH(R)$-modules
\[
THH^{\uR}(A)\simeq THH(A)\sma^{\bL}_{THH(R)}\uR.
\]
\end{prop}

\begin{prop}\label{prop:derrelthh}
Let $\uR'\to \uR$ be an $\aF$-equivalence of cofibrant
$\bT$-equivariant commutative ring orthogonal spectra (in either the
standard model structure or the model structure of~\cite[Thm.~C]{BM-cycl23})
for any family $\aF$ of proper subgroups of $\bT$. Let $A'$ be a cofibrant
associative $R'$-algebra and let $A$ be either a cofibrant associative
$R$-algebra or a cofibrant commutative $R$-algebra.  A weak
equivalence of $R'$-algebras $A'\to A$ then induces an
$\aF$-equivalence of $\bT$-equivariant $\uR'$-modules
$THH^{\uR'}(A')\to THH^{\uR}(A)$.
\end{prop}

When we discuss the derived functor of relative $THH^{\uR}(A)$, we mean the
derived functor of both variables $\uR$ and $A$, which exists by the
previous proposition.  The following definition makes this precise.

\begin{defn}\noindent
\begin{enumerate}
\item 
Let $\aPair_{\aAlg}$ be the category where an object consists of a pair $(\uR,A)$
with $\uR$ a $\bT$-equivariant commutative ring orthogonal spectrum
and $A$ an associative $R$-algebra, and a map $(\uR',A')\to
(\uR,A)$ consists of a map of $\bT$-equivariant commutative ring
orthogonal spectra $\uR'\to \uR$ and a map of associative
$R'$-algebras $A'\to A$ (relative to the given map $\uR'\to \uR$).
\item
Let $\aPair_{\aMod}$ be the category of pairs $(\uR,M)$ with $\uR$ a
$\bT$-equivariant commutative ring orthogonal spectrum and $M$ a
$\bT$-equivariant $\uR$-module with the evident maps (analogous to the
definition of $\aPair_{\aAlg}$). 
\item
For $\aF$ a family of subgroups
of $\bT$, we let the $\aF$-equivalences in $\aPair_{\aAlg}$ be the maps
$(\uR',A')\to (\uR,A)$ where $\uR'\to \uR$ is an $\aF$-equivalence and
$A'\to A$ is a weak equivalence; we let the $\aF$-equivalences in
$\aPair_{\aMod}$ be the maps $(\uR',M')\to (\uR,M)$ where $\uR'\to \uR$
and $M'\to M$ are $\aF$-equivalences.
\end{enumerate}
\end{defn}

\begin{ter}\label{ter:derthh}
By the \term{derived functor of relative $THH$}, we mean the total
left derived functor of $(\uR,A)\mapsto THH^{\uR}(A)$ as a functor
from $\aPair_{\aAlg}$ to $\aPair_{\aMod}$ with the
$\aF_{p}$-equivalences, and we denote it as $\LTHH$. 
\end{ter}

Since $\LTHH$ is defined as a derived functor in a category of pairs
$(\uR,A)$, it is worth commenting on it as a functor of $A$ when we
keep $\uR$ fixed.  The functor $A\mapsto (\uR,A)$ from associative
$R$-algebras to $\aPair_{\aAlg}$ preserves weak equivalences and so
induces a functor on homotopy categories; this gives a restriction of
$\LTHH^{\uR}(-)$ to a functor from the homotopy category of associative
$R$-algebras to the homotopy category of $\aPair_{\aMod}$.  
Since a weak equivalence of $\bT$-equivariant
commutative ring orthogonal spectra $\uR'\to \uR$ induces an
equivalence of $\aF_{p}$-homotopy categories of modules, we can refine
$\LTHH^{\uR}(-)$ to a functor from the homotopy category of associative
$R$-algebras to the homotopy category of $\bT$-equivariant
$\uR$-modules (which we still write as $\LTHH^{\uR}(-)$ by abuse of
notation).  When $\uR$ is cofibrant
in either the standard model structure on $\bT$-equivariant
commutative ring orthogonal spectra or the model structure
of~\cite[Thm.~C]{BM-cycl23}, this restriction
$\LTHH^{\uR}(-)$ agrees with the total left derived functor of
$THH^{\uR}(-)$ from associative $R$-algebras to $\bT$-equivariant
$\uR$-modules (with the $\aF_{p}$-equivalences).  As a consequence, as
long as we assume $\uR$ is cofibrant, we do not have to worry about
the difference between deriving relative $THH$ in both the base and
algebra variable or in just the algebra variable.  To preclude any
confusion, we will usually take $\uR$ to be cofibrant in statements
involving the derived functor of relative $THH$.

\section{Cyclotomic and precyclotomic structures}\label{sec:cyclotomic}
\label{sec:power}

In this section, we review the definitions of $p$-cyclotomic and
$p$-precyclotomic structures on $\bT$-equivariant orthogonal spectra. We
review the theory of commutative ring objects in these categories and
the universal property of $THH$ in this setting.  We define the
$p$-cyclotomic power operation $\Psi$ for commutative ring $p$-precyclotomic
spectra and compare it to the Adams operation $\psi^{p}$ on $THH$.

\begin{defn}
A $p$-precyclotomic spectrum $\cX$ consists of a $\bT$-equivariant
orthogonal spectrum $\uX$ and a map of $\bT$-equivariant
orthogonal spectra $r\colon \Phi \uX\to \uX$.  A $p$-cyclotomic spectrum
is a $p$-precyclotomic spectrum where the induced map 
\[
\bL\Phi \uX\to \Phi \uX\overto{r}\uX
\]
in the $\bT$-equivariant stable category is an
$\aF_{p}$-equivalence. A map of $p$-precyclotomic spectra $\cX\to\cY$ consists of
a map of the underlying $\bT$-equivariant orthogonal spectra $f\colon \uX\to \uY$
that makes the $p$-precyclotomic structure maps commute.
\[
\xymatrix{%
\Phi \uX\ar[r]^{r_{\cX}}\ar[d]_{\Phi f}&\uX\ar[d]^{f}\\
\Phi \uY\ar[r]_{r_{\cY}}&\uY
}
\]
A map of $p$-cyclotomic spectra is a map of $p$-precyclotomic
spectra. 
\end{defn}

The papers~\cite{BM-cycl} and~\cite{ABGHLM} consider the more
sophisticated structure of precyclotomic spectra; however, we work
exclusively in the $p$-precyclotomic setting and therefore simplify
terminology: 

\begin{ter}
For the purposes of this paper, we write \textbf{precyclotomic} for
$p$-precyclotomic and \textbf{cyclotomic} for $p$-cyclotomic.  We
write (pre)cyclotomic to handle both cases together, with an implicit
``respectively''. 
\end{ter}

Because the lax symmetric monoidal structure on $\Phi$ is not (known
to be) strong, we do not have a symmetric monoidal structure on
precyclotomic spectra that refines the usual smash product on
$\bT$-equivariant orthogonal spectra.  As a consequence, we cannot do
the usual thing and define associative ring precyclotomic spectra as
monoids for the smash product.  Instead we
follow~\cite[\S III.2]{BM-cycl23} to define ring structures as follows.  For
the following definition we note that for a $\bT$-equivariant associative ring
orthogonal spectrum $\uA$, $\Phi \uA$ inherits the structure of a
$\bT$-equivariant 
associative ring orthogonal spectrum using the lax symmetric monoidal
structure maps for $\Phi$: it has unit and product
\[
\bS\overto{\iota} \Phi \bS\overto{\Phi \eta} \Phi \uA, \qquad 
\Phi \uA\sma\Phi \uA\overto{\lambda}\Phi (\uA\sma \uA)\overto{\Phi \mu}\Phi \uA
\]
where $\eta$ and $\mu$ are the unit and product for $\uA$.

\begin{defn}[{\cite[\S III.2]{BM-cycl23}}]\noindent
\begin{enumerate}
\item 
An \term{associative ring precyclotomic spectrum} consists of a
precyclotomic spectrum $\cA$ together with a $\bT$-equivariant
associative ring orthogonal spectrum structure on $\uA$ such that the
structure map is a map of $\bT$-equivariant associative ring
orthogonal spectra. 
A map of
associative ring precyclotomic spectra is a map of precyclotomic
spectra that on the underlying $\bT$-equivariant orthogonal spectra is
a map of $\bT$-equivariant associative ring orthogonal spectra.
\item
A \term{commutative ring precyclotomic spectrum}
is an associative ring precyclotomic spectrum whose underlying
$\bT$-equivariant associative ring orthogonal spectrum is
commutative. The
category of commutative ring precyclotomic spectra is a full
subcategory of the category of associative ring precyclotomic spectra.
\item
Commutative and associative ring cyclotomic spectra are
commutative and associative ring precyclotomic spectra (respectively)
whose underlying precyclotomic spectra are cyclotomic.
The categories of associative ring
cyclotomic spectra and commutative ring cyclotomic spectra are 
full
subcategories of the category of associative ring precyclotomic spectra.
\item In any of the categories above, a \term{weak equivalence} is a
map that is an $\aF_{p}$-equivalence of the underlying
$\bT$-equivariant orthogonal spectra.
\end{enumerate}
\end{defn}

For an associative or commutative ring precyclotomic spectrum $\cA$, we
have corresponding notions of (pre)cyclotomic $\cA$-modules.  When $\uM$
is an $\uA$-module, $\Phi \uM$ obtains a canonical $\Phi \uA$-module
structure; we use this in the following definition.  

\begin{defn}\label{defn:precycmod}
Let $\cA$ be an associative ring precyclotomic spectrum.
A (pre)\-cyclotomic $\cA$-module is a (pre)cyclotomic spectrum
$\cM$, together with the structure of an $\uA$-module on $\uM$ making the following
action diagram commute.
\[
\xymatrix{%
\Phi \uA\sma \Phi \uM\ar[d]_{r_{\cA}\sma r_{\cM}}\ar[r]&\Phi \uM\ar[d]^{r_{\cM}}\\
\uA\sma \uM\ar[r]&\uM.
}
\]
A map of (pre)cyclotomic $\cA$-modules is a map of precyclotomic spectra
that is also a map of $\uA$-modules.
\end{defn}

When the underlying $\bT$-equivariant orthogonal spectrum of $\cA$ is in
the class $\EQC$ of Section~\ref{sec:equivariant} (for example, when
$\uA$ is cofibrant as a $\bT$-equivariant 
commutative or associative ring spectrum in any of the model
categories we consider), the smash product with $\uA$ monad
$\uA\sma(-)$ on $\bT$-equivariant orthogonal spectra lifts to a monad
$\cA\sma(-)$ on
precyclotomic spectra, and a precyclotomic $\cA$-module is precisely
an algebra over this monad as usual.

We concentrate on the case of commutative ring precyclotomic spectra
and we recall from~\cite[Thm.~B]{BM-cycl23} a ``shortcut'' for
describing (up to weak equivalence) the mapping space in this category in
terms of mapping spaces in $\bT$-equivariant commutative ring
orthogonal spectra.  Given commutative ring precyclotomic spectra $\cA$
and $\cB$, the space of maps of commutative ring precyclotomic spectra
$\aCom^{\PreCyc}(\cA,\cB)$ can be identified as the equalizer of
\[
\xymatrix@C-1pc{%
\aCom^{\bT}(\uA,\uB)\ar@<-.5ex>[r]\ar@<.5ex>[r]
&\aCom^{\bT}(\Phi \uA,\uB)
}
\]
where one map in the system takes the map $f\colon \uA\to \uB$ to $f\circ
r_{\cA}$ and the other takes it to $r_{\cB}\circ \Phi f$.

To translate this into a homotopical result, we use the model
structure on the category of commutative ring precyclotomic spectra
of~\cite[Thm.~D]{BM-cycl23}.  The fibrations and weak equivalences in this
structure are created by the forgetful functor to precyclotomic
spectra, which in turn are created by the forgetful functor to a variant
$\aF_{p}$-model structure on $\bT$-equivariant orthogonal spectra.
For a
cofibrant commutative ring precyclotomic spectrum $\cA$, the underlying
$\bT$-equivariant commutative ring spectra $\uA$ and $\Phi \uA$ are
cofibrant. When in addition $\cB$ is fibrant, the map from the
equalizer above to the corresponding homotopy equalizer is a weak
equivalence and~\cite[Thm.~B]{BM-cycl23} gives the following shortcut to
computing the derived mapping spaces.

\begin{prop}[{\cite[Thm.~B]{BM-cycl23}}]\label{prop:mappingspace}
Let $\cA$, $\cB$ be commutative ring precyclotomic spectra, and assume
that for $\cA$ the underlying $\bT$-equivariant commutative ring
orthogonal spectrum is cofibrant and for $\cB$ the underlying
$\bT$-equivariant commutative ring orthogonal spectrum is fibrant in
the model structure of~\cite[Thm.~C]{BM-cycl23}.  Then the
derived mapping space
$\bR\aCom^{\PreCyc}(\cA,\cB)$ of commutative ring precyclotomic spectra
maps from $\cA$ to $\cB$ is represented by the homotopy equalizer
of the maps
\[
\xymatrix@C-1pc{%
r_{\cA}^{*},r_{\cB*}\circ \Phi\colon \aCom^{\bT}(\uA,\uB)\ar@<-.5ex>[r]\ar@<.5ex>[r]
&\aCom^{\bT}(\Phi \uA,\uB).
}
\]
\end{prop}

Specializing to the case when $\cA=THH(R)$ for a cofibrant commutative
ring orthogonal spectrum $R$, not only is the underlying $\bT$-equivariant
commutative ring orthogonal spectrum $\uA$ cofibrant, but the
precyclotomic structure map $\Phi \uA\to \uA$ is an isomorphism.  Using
this isomorphism, the derived mapping space $\bR\aCom^{\PreCyc}(\cA,\cB)$
is represented by the homotopy equalizer of the maps
\[
\xymatrix@C-1pc{%
\id,r_{\cA}^{*-1}\circ (r_{\cB*}\circ \Phi) \colon \aCom^{\bT}(\uA,\uB)\ar@<-.5ex>[r]\ar@<.5ex>[r]
&\aCom^{\bT}(\uA,\uB).
}
\]
The universal property of Proposition~\ref{prop:tensor} lets us identify the derived
mapping space $\bR\aCom^{\PreCyc}(\cA,\cB)$ in this case as a homotopy equalizer 
\begin{equation}\label{eq:mapsthh}
\xymatrix@C-1pc{%
\id,\varsigma \colon \aCom(R,B)\ar@<-.5ex>[r]\ar@<.5ex>[r]
&\aCom(R,B)
}
\end{equation}
for some map $\varsigma$, which we now describe.  It requires a
structure on commutative ring precyclotomic spectra that we call the
cyclotomic power operation. 

\begin{defn}\label{def:power-operation}
Let $\cB$ be a commutative ring precyclotomic spectrum.  The \term{cyclotomic
power operation} is the composite
\[
\Psi \colon B \to \Phi^{C_p} N_e^{C_p} B \to \Phi^{C_p} \uB \overto{r_{\cB}} B,
\]
where the first map is the diagonal, the second map is $\Phi^{C_p}$
applied to the map~\eqref{eq:counit}, and the third map is
the precyclotomic structure map.
\end{defn}

\begin{prop}\label{prop:Psi}
The map $\varsigma$ in~\eqref{eq:mapsthh} is the map $\Psi_{*}$ given
by post-composition with the cyclotomic power operation on $\cB$.
\end{prop}

\begin{proof}
Given a map $f\in \aCom(R,B)$, consider the diagram in
(non-equivariant) commutative ring orthogonal spectra
\begin{equation}\label{eq:pfAdams}
\begin{gathered}
\xymatrix{%
R\ar[d]\ar[r]^-{\iso}
  \ar@/_3em/[dd]_-{f}&\Phi^{C_{p}}(N_{e}^{C_{p}}R)\ar[d]\ar@{.>}[dr]
\\
THH(R)\ar[r]^-{\iso}\ar[d]^{\tilde f}
&\Phi^{C_{p}}(N_{e}^{C_{p}}THH(R))
  \ar[d]^{\Phi^{C_{p}}N_{e}^{C_{p}}\tilde f}\ar[r]
&\Phi^{C_{p}}THH(R)\ar[d]^{\Phi^{C_{p}}\tilde f}\\
B\ar[r]&\Phi^{C_{p}}(N_{e}^{C_{p}}B)\ar[r]&\Phi^{C_{p}}\uB
}
\end{gathered}
\end{equation}
where $\tilde f$ is the underlying non-equivariant map of the map
$THH(R)\to \uB$ adjoint to $f$
under Proposition~\ref{prop:tensor} and in the bottom two rows, the maps are the
first two maps in the definition of $\Psi$.  In the top row the
isomorphism is the diagonal.  The solid part of the diagram then
commutes, and if we fill in the dotted map $R\to \Phi^{C_{p}}THH(R)$
as the map induced on $\Phi^{C_{p}}$ by the inclusion
$N_{e}^{C_{p}}R\to THH(R)$, the whole diagram commutes.  When we
interpret the top isomorphism followed by the dotted map as a
(non-equivariant) map $R\to \Phi 
THH(R)$, the adjoint map 
\[
THH(R)\to \Phi THH(R)
\]
under Proposition~\ref{prop:tensor} is the map $r_{\cA}^{-1}$.  Viewing the overall
composite map $R\to \Phi^{C_{p}}\uB$ as a (non-equivariant) map $R\to
\Phi \uB$, the adjoint map 
\[
THH(R)\to \Phi \uB
\]
is then $\Phi \tilde f \circ r_{\cA}^{-1}$. 
\end{proof}

The diagram in the proof above also identifies the cyclotomic power
operation $\Psi$ on $THH(R)$.

\begin{prop}\label{prop:PsiTHH}
Let $R$ be a cofibrant commutative ring orthogonal spectrum.  Then $\Psi \colon
THH(R)\to THH(R)$ is the Adams operation 
given by tensoring $R$ with the $p$-fold covering map on $\bT$ (see,
for example, \cite[\S 10]{ABGHLM}).
\end{prop}

\begin{proof}
The proof of Proposition~\ref{prop:Psi} presents a
diagram~\eqref{eq:pfAdams} in which the composite map $R\to
\Phi^{C_{p}}THH(R)$ is the map adjoint (under
Proposition~\ref{prop:tensor}) to $r^{-1}_{THH(R)}$, the inverse of
the cyclotomic structure map for $THH(R)$.  The operation $\Psi$ on
$THH(R)$ is the composite of the middle row of that
diagram $THH(R)\to \Phi^{C_{p}}THH(R)$ with the cyclotomic
structure map $r_{THH(R)}$.  It follows that the composite of the
inclusion $R\to THH(R)$ with the operation $\Psi$ is again the
inclusion $R\to THH(R)$. If the operation $\Psi$ were equivariant, the
adjunction of Proposition~\ref{prop:tensor} would identify it as the
identity map; however, it is not equivariant for the usual
$\bT$-action on $THH$.  The proof of the
statement is essentially a careful check that it is equivariant when
the target is given the $\bT$-action pulled back from the $p$-fold
covering map $\bT\to \bT$. 

Working on the point-set level, $\bT$-equivariantly, $THH(R)$ is given
by the cyclic bar construction followed by change of universe 
\[
THH(R)=I_{\bR^{\infty}}^{U}N^{\cy}(R)=I_{\bR^{\infty}}^{U}(R\otimes \bT).
\]
(As functors from commutative ring orthogonal spectra to
$\bT$-equivariant orthogonal spectra indexed on $\bR^{\infty}$, we can
identify the cyclic bar construction as the tensor with $\bT$, and we
use these descriptions interchangeably.)  Working non-equivariantly is
in particular working in the universe $\bR^{\infty}$; keeping track of
universes, the first map in $\Psi$, the diagonal map, is
\[
N^{\cy}(R)\overto{\iso}\Phi^{C_{p}}(N_{e}^{C_{p}}N^{\cy}(R)).
\]
We recall that $N^{C_{p}}_{e}$ is a continuous point-set functor from non-equivariant
orthogonal spectra indexed on $\bR^{\infty}$ to $C_{p}$-equivariant orthogonal
spectra indexed on the complete universe $U_{C_{p}}$ (obtained by
restricting the $\bT$-action on $U$ to $C_{p}$), and $\Phi^{C_{p}}$ is
a continuous point-set functor on the same categories in the opposite direction.
The map above is natural in maps on $N^{\cy}(R)$ in (non-equivariant) orthogonal
spectra and so is a map of $\bT$-equivariant
orthogonal spectra indexed on $\bR^{\infty}$ (using the 
$\bT$-action on $N^{\cy}(R)$). For the next map in $\Psi$, written
\[
\Phi^{C_{p}}N_{e}^{C_{p}}THH(R)\to\Phi^{C_{p}} THH(R),
\]
we are looking at $THH(R)$ on the right as a $C_{p}$-equivariant
orthogonal spectrum indexed on $U_{C_{p}}$. This map is
$\Phi^{C_{p}}$ applied to the map
\[
N_{e}^{C_{p}}N^{\cy}(R)\to I_{\bR^{\infty}}^{U_{C_{p}}}N^{\cy}(R)
\]
of $C_{p}$-equivariant orthogonal spectra indexed on $U_{C_{p}}$.
Again by naturality, using the action of 
$\bT$ on
$N^{\cy}(R)$ in the category of non-equivariant orthogonal spectra,
the displayed map is $\bT$-equivariant in the category of
$C_{p}$-equivariant orthogonal spectra indexed on $U_{C_{p}}$.
The composite
\begin{equation}\label{eq:composite}
N^{\cy}(R)\overto{\iso}\Phi^{C_{p}}(N_{e}^{C_{p}}N^{\cy}(R))\to 
\Phi^{C_{p}}(I_{\bR^{\infty}}^{U_{C_{p}}}N^{\cy}(R))
\end{equation}
is a map of $\bT$-equivariant orthogonal spectra indexed on
$\bR^{\infty}$ and so is determined by the restriction of its
underlying map of non-equivariant spectra along the
inclusion $R\to N^{\cy}(R)$, which was analyzed in~\eqref{eq:pfAdams}.
Rewritten to emphasize the universes, it is the map
\begin{equation}\label{eq:restreq}
R\overto{\iso}\Phi^{C_{p}}(N_{e}^{C_{p}}R)\to \Phi^{C_{p}}(I_{\bR^{\infty}}^{U_{C_{p}}}N^{\cy}(R))
\end{equation}
induced by the map
\[
N_{e}^{C_{p}}R = I_{\bR^{\infty}}^{U_{C_{p}}}(R\otimes C_{p})\to 
I_{\bR^{\infty}}^{U_{C_{p}}}(R\otimes \bT)=I_{\bR^{\infty}}^{U_{C_{p}}}N^{\cy}(R).
\]

Because the target of~\eqref{eq:restreq} is $C_{p}$-fixed, we can
treat it as a $\bT/C_{p}$-equivariant orthogonal spectrum indexed on
$\bR^{\infty}$ and the identification of~\eqref{eq:restreq} above
shows that under the isomorphism $\rho\colon \bT\iso \bT/C_{p}$, the adjoint map
\[
R\otimes (\bT/C_{p})\to \Phi^{C_{p}}(I_{\bR^{\infty}}^{U_{C_{p}}}N^{\cy}(R))
\]
becomes the underlying non-equivariant map of $r_{THH(R)}^{-1}$.  It
follows that the map~\eqref{eq:composite} is the composite of the
$p$-fold cover map $\bT\to \bT$ and $r_{THH(R)}^{-1}$.  Thus,
composing with the underlying non-equivariant map of $r_{THH(R)}$,
$\Psi$ is the map induced by the $p$-fold cover.
\end{proof}

\section{Relative cyclotomic structures}\label{sec:relcyc}

The purpose of this section is to explain our new framework for the
existence of cyclotomic structures on relative $THH$.  The work
of~\cite{ABGHLM} shows that for a commutative ring (pre)cyclotomic
spectrum $\cR$ and an $R$-algebra $A$, $THH^{\uR}(A)$
obtains a natural (pre)cyclotomic structure precisely when the
canonical map of $\bT$-equivariant commutative ring spectra $THH(R)
\to \uR$ is a (pre)cyclotomic map on the point-set level.  In this
section, we generalize this to the case when the map $THH(R)\to \uR$
lifts to a map of commutative ring (pre)cyclotomic spectra in the
homotopy category.  Our framework depends on the characterization in
the previous section of derived mapping spaces by allowing us to
specify in terms of the cyclotomic power operation when the canonical
map $THH(R) \to \uR$ lifts.

Specifically, let $\cR$ be a commutative ring
precyclotomic spectrum, and assume without loss of generality that
$\cR$ is cofibrant and fibrant in that category for the model
structure of~\cite[Thm.~D]{BM-cycl23}.  Then the
underlying non-equivariant commutative ring orthogonal spectrum $R$
is cofibrant and fibrant in the model structure
of~\cite[III.4.1]{BM-cycl23} 
(by~\cite[III.1.3]{BM-cycl23}).
In this case, Proposition~\ref{prop:Psi} specializes to the following
result.

\begin{prop}\label{prop:dmapTHH}
Let $\cR$ be a cofibrant-fibrant commutative ring precyclotomic
spectrum.  Then the derived mapping space
$\bR\aCom^{\PreCyc}(THH(R),\cR)$ is modeled by the homotopy equalizer
of the self-maps 
\[
\id,\Psi_{*}\colon \aCom(R,R)\to \aCom(R,R). 
\]
\end{prop}

In particular, 
the canonical map of $\bT$-equivariant commutative ring orthogonal
spectra $THH(R)\to \uR$ lifts (up to homotopy) to a map 
of commutative ring precyclotomic spectra $THH(R) \to \cR$ exactly
when $\Psi$ is homotopic to the identity.  We encapsulate this in the
following definition.

\begin{defn}
A \term{(pre)cyclotomic base} is a commutative ring (pre)cyclotomic spectrum
$\cR$ together with a choice of homotopy from $\Psi$ to the identity in
the category of commutative ring orthogonal spectra.  
\end{defn}

The following is the main foundational theorem of the paper and gives
a more specific formulation of Theorem~A of the introduction. 

\begin{thm}\label{thm:precycstructure}
Let $\cR$ be a (pre)cyclotomic base, whose underlying
$\bT$-equivariant commutative ring orthogonal spectrum is cofibrant in
the model structure of~\cite[Thm.~C]{BM-cycl23} (for example, when $\cR$ is
cofibrant in the model structure of~\cite[Thm.~D]{BM-cycl23} on commutative
ring (pre)cyclotomic spectra).  Then the derived relative $THH$
functor $\LTHH^{\cR}$ lifts to a functor from associative $R$-algebras
to (pre)cyclotomic $\cR$-modules.
\end{thm}

\begin{proof}
Replacing $\cR$ by a weakly equivalent commutative ring
(pre)cyclotomic spectrum if necessary, we can assume without loss of
generality that $\cR$ is fibrant in the model structure of~\cite[Thm.~D]{BM-cycl23}.
By Proposition~\ref{prop:derrelthh}, for $A$ cofibrant as an associative
$R$-algebra, $\LTHH^{\cR}(A)$ is represented by the point-set construction
\[
THH^{\cR}(A) = THH(A)\sma_{THH(R)}\uR \simeq THH(A)\sma^{\bL}_{THH(R)}\uR.
\]
By the hypothesis on $\Psi$ and Proposition~\ref{prop:Psi}, the
discussion around~\eqref{eq:mapsthh} implies that the diagram in the
category of $\bT$-equivariant commutative ring orthogonal spectra
\[
\xymatrix{%
\Phi THH(R)\ar[r]\ar[d]_{r_{THH(R)}}&\Phi \uR\ar[d]^{r_{\cR}}\\
THH(R)\ar[r]&\uR
}
\]
commutes up to (the given) homotopy. (This specifies an element of the
homotopy equalizer in Proposition~\ref{prop:mappingspace} for
$\cA=THH(R)$ and $\cB=\cR$.) While $THH(R)$ is cofibrant in
the category of $\bT$-equivariant commutative ring orthogonal spectra,
it is generally not cofibrant in the category of commutative ring
cyclotomic spectra. Let $\cT\to THH(R)$ be a cofibrant
approximation in the latter category; the corresponding diagram in
$\uT$ and $\uR$ then also commutes 
up to (the restriction of the given) homotopy, specifying an element
of the homotopy equalizer in Proposition~\ref{prop:mappingspace} (for
$\cA=\cT$ and $\cB=\cR$).  Since $\uT$ is cofibrant
and $\uR$ is fibrant, Proposition~\ref{prop:mappingspace} implies that the
composite map of $\bT$-equivariant commutative ring orthogonal spectra
$\uT\to \uR$ is homotopic to a precyclotomic map $g\colon \cT\to \cR$: the map
from the equalizer to the homotopy equalizer in
Proposition~\ref{prop:mappingspace} is
a weak equivalence.  The space of choices of such a $g$ together with a path $H$ in the
homotopy equalizer from $g$ to the point specified above is weakly
contractible, and we choose an element $(g,H)$. As a component of the 
path $H$, we get a homotopy $G$ from the composite map $\cT\to
THH(R)\to \cR$ to $g$.  We write $g^{*}\uR$ for $\uR$ with the
$\bT$-equivariant commutative $\uT$-algebra structure from the map of $\bT$-equivariant
commutative ring orthogonal spectra $g$, and we use $G$ to give $\uR$
a $\bT$-equivariant $\uT\otimes I$-module structure.  We then have a
zigzag of weak equivalences
\[
THH(A)\sma^{\bL}_{THH(R)}\uR\overfrom{\simeq}
THH(A)\sma^{\bL}_{\uT}\uR\overto{\simeq}
THH(A)\sma^{\bL}_{\uT\otimes I}\uR\overfrom{\simeq}
THH(A)\sma^{\bL}_{\uT}g^{*}\uR.
\]
Let $\cM$ be a cofibrant approximation of $g^{*}\cR$ in the model structure on (pre)cyclotomic $\cT$-modules of \cite[III.2.5]{BM-cycl23}.
Then $(-)\sma_{\uT} \uM$ represents the derived functor
$(-)\sma^{\bL}_{\uT}g^{*}\uR$ and $THH(-)\sma_{\uT}\uM$ is another point-set
model for $\LTHH^{\cR}(-)$. 

By construction, $\uT$ is in the class of objects $\EQC$ discussed in
Section~\ref{sec:equivariant} on which $\Phi$ and the smash product
behave well.  By~\cite[III.1.8]{BM-cycl23}, for any $\bT$-equivariant $\uT$-module $\uN$,
the canonical map  
\[
\Phi \uN\sma_{\Phi \uT}\Phi \uM\to \Phi (\uN\sma_{\uT}\uM)
\]
is an isomorphism.  For $\uN=THH(A)$, using the precyclotomic structure
maps on $THH(A)$, $\cT$, and $\cM$, we get a precyclotomic structure map
\[
\Phi (THH(A)\sma_{\uT}\uM)\iso \Phi THH(A)\sma_{\Phi \uT}\Phi \uM\to
THH(A)\sma_{\uT} \uM.
\]
For any family $\aF$ of proper subgroups of $\bT$, the composite map 
\[
\bL\Phi (THH(A)\sma_{\uT}\uM)\to THH(A)\sma_{\uT} \uM
\]
is an $\aF$-equivalence whenever the precyclotomic structure map of $\cM$ is.
\end{proof}

As a consequence, we can define relative topological cyclic homology
in this context.  Since we are concentrating on $p$-(pre)cyclotomic
spectra, the relevant version of topological cyclic homology is
$p$-typical $TC$, denoted as $TC(-;p)$ in~\cite{BM-cycl} (and elsewhere).  To avoid
ambiguity between $TC$ of a precyclotomic object and $TC$ of a ring
spectrum, we use the following notation.

\begin{notn}
We write $TC_{\cyc}$ for the composite functor $TC(R(-);p)$ from
precyclotomic spectra to spectra, where $TC(-;p)$ is the $p$-typical
$TC$-construction of~\cite[6.3]{BM-cycl} (for example) and $R$ is a
fibrant approximation functor in the category of precyclotomic
spectra.  We write $TC$ for the functor $TC_{\cyc}(THH(-))$ from
associative or commutative ring orthogonal spectra to orthogonal
spectra.
\end{notn}

The construction of $TC(-;p)$ in the previous paragraph uses the maps between categorical fixed points
\[
\mathrm{R},\mathrm{F}\colon \uX^{C_{p^{n+1}}}\to \uX^{C_{p^{n}}}
\]
where $\mathrm{F}$ is the inclusion and $\mathrm{R}$ is the composite
of the canonical map
\[
\uX^{C_{p^{n+1}}}\iso (\rho^{*}(\uX^{C_{p}}))^{C_{p^{n}}}\to (\Phi
\uX)^{C_{p^{n}}}\to \uX^{C_{p^{n}}}
\]
where the middle map is the map of~\cite[4.4]{MM} from the categorical
fixed points to the geometric fixed points and the last map is the
$C_{p^{n}}$ fixed points of the precyclotomic structure map $r_{\cX}$.
We note that the construction only requires a precyclotomic structure
and not a cyclotomic structure.  For the homotopically correct
construction of $p$-typical $TC$, we need the homotopically correct categorical
fixed points, which we ensure by fibrant approximation of $\uX$ in the
category of precyclotomic spectra in the definition of $TC_{\cyc}$.

In the case when the underlying $\bT$-equivariant spectrum $\uX$ is
$p$-complete (i.e., all of the categorical fixed point spectra are
$p$-complete as non-equivariant spectra), $TC_{\cyc}(\cX)$ is just the derived
mapping spectrum of maps out of the sphere spectrum in the category of
(pre)cyclotomic spectra \cite[6.8]{BM-cycl}:
\[
TC_{\cyc}(\cX) \iso \bR F^\PreCyc(\bS,\cX).
\]

Returning to the case of relative $THH$, we use the following notation
for $TC$ of these (pre)cyclotomic spectra.

\begin{notn}
Let $\cR$ be a (pre)cyclotomic base.  Write $TC^{\cR}$ for derived
$TC_{\cyc}$ of derived $THH^{\uR}$ with the (pre)cyclotomic structure of
Theorem~\ref{thm:precycstructure}. 
\end{notn}

While this adequately constructs a derived functor $TC^{\cR}$ (from the homotopy
category of $R$-algebras to the stable category), the constructions in
Section~\ref{sec:descent} require a point-set model that is functorial
in $\cR$ as well as $A$.  We describe such a model in
Section~\ref{sec:TCfunc}. 

The construction in Theorem~\ref{thm:precycstructure} of the
precyclotomic structure depended on the choice of homotopy in the
precyclotomic base structure on $\cR$.  We observe that homotopic
homotopies construct weakly equivalent precyclotomic structures.  In
the following proposition, let $D$ denote the $2$-disk obtained as the
(unreduced) cone on two copies of $[0,1]$ with corresponding endpoints
identified. 

\begin{prop}\label{prop:disk}
Let $\cR'$, $\cR''$ be precyclotomic bases with the same underlying
commutative ring precyclotomic spectrum $\cR$.  A map of
commutative ring orthogonal spectra $R\otimes D\to R$ which
restricts on each copy of $[0,1]$ to the homotopy intrinsic to the
precyclotomic base structures on $\cR'$ and $\cR''$ induces a natural
isomorphism in the homotopy category of precyclotomic spectra 
\[
\bL THH^{\cR'}(-)\simeq \bL THH^{\cR''}(-).
\]
\end{prop}
 
\section{Precyclotomic bases, equivariant factorization homology, and
global commutative ring spectra with multiplicative
deflations}\label{sec:facthom}

The purpose of this section is to give a conceptual explanation of the
motivation behind the definition of a precyclotomic base. The
discussion is purely motivational and should not be regarded as
rigorously justified.  The material here was inspired by an October
2021 talk given
by Asaf Horev at MIT.  Horev discussed the structure
of equivariant factorization homology and its relationship to the
cyclotomic structure on $THH$. Extending these observations from
spectra to categories of
$R$-modules and examining the required structure on $R$ leads to the framework of precyclotomic bases, as
we explain below.  Abstracting the categorical framework using ideas
of Bachmann-Hoyois~\cite[\S9]{BachmannHoyois} (see also the third
author's work in~\cite[\S2]{Yuan-Frob}) gives rise to the notion of global
commutative ring spectrum with multiplicative deflations
(Definition~\ref{defn:global}).  This structure arises in nature on
the global equivariant Thom spectrum $\uMUP$ and on a new global
equivariant structure on $MU$, motivating the key examples of precyclotomic
bases that we discuss in more detail in the next section.

As we progress in this section, we will require $\infty$-categorical
constructions, and statements should be read in $\infty$-categorical
terms.  In particular, functors in this section should be read 
in their homotopical rather than point-set forms.

We begin by recalling the relevant parts of the standard structure of
genuine equivariant factorization homology.  Let $G$ be a finite group
and for simplicity, let $\underline{A}$ be a genuine $G$-equivariant
commutative ring orthogonal spectrum with underlying (non-equivariant)
commutative ring orthogonal spectrum $A$.  Then equivariant
factorization homology associates to each $G$-manifold $M$ a genuine
$G$-equivariant orthogonal spectrum $\int_M \underline{A}$ with
the following features (among others):
\begin{enumerate}
    \item\label{efhnorm} When $M = G/H$ is a transitive $G$-set, there is a canonical
    equivalence $\int_M \underline{A} \simeq N_H^G \res^G_H \underline{A}$.
    In particular, $\int_{G/e} \underline{A} \simeq N^G_{e} A$ depends only
    on the underlying non-equivariant commutative ring $A$.  
    \item\label{efhfree} When $M$ has a free action of $G$, there is a natural equivalence 
    \[
    \Phi^G \int_M \underline{A} \simeq \int_{M/G} A.
    \]
\end{enumerate}
(In the more general case when $\underline{A}$ is some kind of
equivariant disk algebra, representation tubes of the form
$G\times_{H}V$ should replace orbits in~(\ref{efhnorm}).)

Two remarks are in order regarding the features above: first,
in the special case $M = G/e$, the equivalence in~(\ref{efhfree})
extends the diagonal 
identity $\Phi^G N^G_{e} A \simeq A$, which is the case $M=G$.  Second, 
the observation about restriction in~(\ref{efhnorm}) generalizes to~(\ref{efhfree}):
when $M$ is free, $\int_{M}\underline{A}$ depends only on the
non-equivariant commutative ring orthogonal spectrum $A$. Indeed,
features~(\ref{efhnorm}) and~(\ref{efhfree}) are closely related, 
with~(\ref{efhfree}) deriving from~(\ref{efhnorm}).

The above features relate to the cyclotomic structure on $THH$ in the
following way.  We take the manifold $M$ to be $\bT$ and $G$ to be a
finite subgroup $C_{n}$ of $\bT$.  The genuine $C_{n}$-equivariant
orthogonal spectrum $\int_{\bT} \underline{A}$
comes with a compatible $\bT$-action extending the inherent
$C_{n}$-action.  As the subgroups $C_{n}$ vary, they fit together to
produce a genuine $\bT$-equivariant $\aF_{\fin}$-colocal orthogonal
spectrum.  Since $\bT$ is a free $C_{n}$-manifold, we also have
(non-equivariant) equivalences 
\[
\Phi^{C_{n}}\int_{\bT}A \simeq \int_{\bT/C_{n}} A,
\]
natural in $C_{n}$-equivariant self-maps of $\bT$, and in
particular natural in the $\bT/C_{n}$-action on both sides.
Applying the Borel equivariant version of
$\rho_{n}^{*}$ for $\rho_{n}\colon \bT\iso \bT/C_{n}$ the $n$th root map,
we can view the above map as a Borel equivalence of $\bT$-spectra
\[
\rho_{n}^{*}\Phi^{C_{n}}\int_{\bT}A \simeq
\rho_{n}^{*}\int_{\bT/C_{n}} A \iso \int_{\bT} A.
\]
Looking at all the $C_{n}$ together (assuming the equivalences are
appropriately compatible with inclusions of subgroups), we deduce a
genuine $\bT$-equivariant $\aF_{\fin}$-equivalence
\[
\Phi \int_{\bT}A \simeq \int_{\bT}A.
\]
Thus, the cyclotomic structure on $THH$ is a consequence of the basic
features of equivariant factorization homology.

In order to apply this observation to $THH$ relative to a commutative
ring orthogonal spectrum $R$, we would replace (equivariant) spectra
with (equivariant) $R$-modules for some equivariant structure $\uR$ on
$R$. To start, we need the category of $R$-modules to admit norm functors
${}_{\uR}N_{e}^{G}$ from (non-equivariant) $R$-modules to
$G$-equivariant $\uR$-modules, and geometric fixed point functors
${}_{\uR}\Phi^{G}$ from $G$-equivariant $\uR$-modules to (non-equivariant)
$R$-modules, related by a natural equivalence
\[
\Id \overto{\simeq}{}_{\uR}\Phi^{G}{}_{\uR}N_{e}^{G}, 
\]
which is symmetric monoidal and appropriately compatible with
restrictions to subgroups (as encapsulated below).

Applying the norm $N_{e}^{G}$ in spectra to an $R$-module $X$ 
naturally yields an $N_{e}^{G}R$-module $N_{e}^{G}X$.  To convert this to an
$\uR$-module, a norm multiplication $N_{e}^{G}R\to \uR$ suffices, and
we can set
\[
{}_{\uR}N_{e}^{G}X:=N_{e}^{G}X\sma_{N_{e}^{G}R}\uR.
\]
If we also assume these norm multiplications satisfy the usual
compatibilities, by~\cite[6.11]{BlumbergHill-NormsTransfers} the
resulting structure essentially amounts to a genuine $G$-equivariant
commutative ring orthogonal spectrum structure on $\uR$ with
underlying non-equivariant commutative ring orthogonal spectrum $R$.

Analogously, applying the geometric fixed point functor $\Phi^G$ in
spectra to an $\uR$-module $\uX$ naturally yields a $\Phi^G \uR$-module
$\Phi^G \uX$. To convert this to an $R$-module, it suffices to have a
map of commutative ring orthogonal spectra $\Phi^{G}\uR\to R$.  That is, we define
\[
{}_{\uR} \Phi^G \uX := \Phi^{G} \uX \sma_{\Phi^{G} \uR} R.
\]
We say
more below about making these fit together along restriction maps,
which leads to equivariance considerations, but for now we note that
for $G=C_{p}$, this is the underlying non-equivariant map of a
precyclotomic structure map.

Now consider the existence of a diagonal map $X\overto{\simeq}
{}_{\uR}\Phi^{G}{}_{\uR}N_{e}^{G}X$.  By definition, the target
functor is the composite
\[
{}_{\uR}\Phi^{G}{}_{\uR}N_{e}^{G} X = 
\Phi^{G}(N_{e}^{G}X\sma_{N_{e}^{G}R}\uR)\sma_{\Phi^{G}\uR}R
   \simeq 
(\Phi^{G}N_{e}^{G}X)\sma_{\Phi^{G}N_{e}^{G}R}
   \Phi^{G}\uR\sma_{\Phi^{G}\uR}R.
\]
Using the diagonal equivalence
for $\Phi^{G}$ and $N_{e}^{G}$ in spectra, this functor is naturally equivalent
to the extension of scalars $X\sma_{R}R$ for the self-map of $R$ given
by the composite 
\[
\Psi \colon R\overto{\simeq}\Phi^{G}N_{e}^{G}R\to \Phi^{G}\uR\to R.
\]
An identification of the composite as the identity in the
$\infty$-category of commutative ring orthogonal spectra then
constructs a natural diagonal equivalence
\[
X\overto{\simeq} {}_{\uR}\Phi^{G}{}_{\uR}N_{e}^{G}X
\]
for $R$-modules.  For $G=C_{p}$, the operation $\Psi$ is precisely the
operation in the definition of precyclotomic base, and (for $\uR$ a
precyclotomic spectrum) the identification in the $\infty$-category is
essentially a choice of homotopy that gives the structure of a
precyclotomic base.

All this was a discussion of norms functors and geometric fixed point
functors in $R$-modules; we now turn to equivariant factorization
homology.  As in our simplification in the discussion of the features
of equivariant factorization homology in spectra, we restrict to the
commutative case.  In this case, equivariant
factorization homology extends to a functor on all $G$-spaces given by
prolongation of norms.  To make this work, extending norms to all
finite $G$-sets by smash product, we need norms to be functorial in
maps of $G$-sets. Assuming enough structure on $\uR$ (including at
least the structure above, more about which below), we can do this for
commutative $\uR$-algebras: for a commutative $\uR$-algebra~$\underline{A}$, 
\[
G/H_{1} \amalg \dotsb \amalg G/H_{n}\mapsto 
({}_{\uR}N_{H_{1}}^{G}\res_{H_{1}}^{G}\underline{A})
  \sma_{\uR}\dotsb\sma_{\uR}
({}_{\uR}N_{H_{n}}^{G}\res_{H_{n}}^{G}\underline{A})
\]
extends to a functor ${}_{\uR}N\underline{A}$ from finite $G$-sets to
$\uR$-modules (or commutative $\uR$-algebras).
We then define 
\[
M\mapsto \int^{\uR}_{M}\underline A
\]
to be the functor (in $M$) from $G$-spaces to $\uR$-modules (or commutative
$\uR$-algebras) given by the left Kan extension of
${}_{\uR}N\underline{A}$ under the inclusion of finite $G$-sets in all
$G$-spaces.  (Because disjoint unions of $G$-sets go to coproducts of
commutative $\uR$-algebras, using commutative $\uR$-algebras as the
target and then forgetting to $\uR$-modules gives a naturally
equivalent functor.)  The diagonal equivalence $A\overto{\simeq}
{}_{\uR}\Phi^{G}{}_{\uR}N_{e}^{G}A$ extends to a natural diagonal equivalence
\[
{}_{\uR}N\underline{A}(S/G)\overto{\simeq}{}_{\uR}
\Phi^{G}({}_{\uR}N\underline{A}(S))
\]
for free $G$-sets $S$ and prolongs to a natural diagonal equivalence
\[
\int^{\uR}_{M/G}\underline{A}\overto{\simeq}
{}_{\uR}\Phi^{G} \bigg(\int^{\uR}_{M}\underline{A}\bigg) 
\]
for free $G$-spaces $M$, with both $\int^{\uR}_{M/G}\underline{A}$ and
$\int^{\uR}_{M}\underline{A}$ depending only on the underlying
non-equivariant commutative $R$-algebra structure on $\underline{A}$.
In other words, this theory has the features~(\ref{efhnorm})
and~(\ref{efhfree}) of equivariant factorization homology discussed
above.  In the case when $M=\bT$ and we let $G$ range over the finite
subgroups $C_{n}<\bT$ as above, with some good will and an extension
of the structure of $\uR$ to $\bT$-equivariant commutative orthogonal
ring spectra, the observations above on $\int_{\bT}$ now generalize to
$\int^{\uR}_{\bT}$: for a commutative $R$-algebra $A$, the
$\aF_{\fin}$-colocal $\bT$-equivariant $\uR$-module $\int^{\uR}_{\bT}A$
comes with an equivalence 
\[
\rho^{*}\left({}_{\uR}\Phi^{C_{p}}\int^{\uR}_{\bT}A\right)\simeq \int^{\uR}_{\bT}A.
\]
We note that this is not necessarily a cyclotomic structure map but
does induce a precyclotomic structure map
\[
\Phi \int^{\uR}_{\bT}A=\rho^{*}\left(\Phi^{C_{p}} \int^{\uR}_{\bT}A\right)\to
\rho^{*}\left({}_{\uR}\Phi^{C_{p}}\int^{\uR}_{\bT}A\right)\simeq \int^{\uR}_{\bT}A.
\]

The discussion above took a direct and streamlined approach to connect
the ideas behind equivariant factorization homology in the category of
$R$-modules and precyclotomic bases. One way to fill in some of the
missing structure is to use the formulation of global equivariant
$E_{\infty}$ ring spectra in \cite[\S 1.4, 9]{BachmannHoyois}
and~\cite[\S 2]{Yuan-Frob}.  Let $\Span(\Gpd)$
denote the $\infty$-category whose objects are finite groupoids (that
is, finite $\pi_1$ and finitely many components) $X$ and where
morphisms between $X$ and $Y$ are given by the space of spans
$X \leftarrow Z \rightarrow Y$, with composition of spans evidenced by
homotopy cartesian squares.

\begin{defn}
A (multiplicative) global equivariant context is a functor $\aC(-): \Span(\Gpd) \to \Cat_{\infty}$ which sends disjoint unions to products.  
\end{defn}

Although we will not give details, global equivariant stable homotopy
theory fits into this framework: there is a global equivariant context
$\underline{\Sp}$ where $\underline{\Sp}(BG)$ is equivalent to the
$\infty$-category of genuine $G$-equivariant orthogonal spectra.  We
note that every map in $\Span(\Gpd)$ is equivalent to a disjoint union
of composites of maps of the form $BG\from BH$, $BK\from BG$, $BH\to
BG$, and $BG\to BK$ for inclusions of subgroups $H\to G$ and quotient
maps of quotient groups $G\to K$.  These component pieces have
classical interpretations in equivariant stable homotopy theory:

\begin{notn}\label{notn:functoriality}
    Given a global equivariant context $\aC$ and a short exact sequence of groups
    $0 \to H \hookrightarrow G \twoheadrightarrow K \to 0$,  we denote:
\begin{align*}
_{\aC}\mathrm{res}^G_H := \aC(BG \leftarrow BH \rightarrow BH) &: \aC(BG) \to \aC(BH)\\
_{\aC}\mathrm{inf}_K^G := \aC(BK \leftarrow BG \rightarrow BG) &: \aC(BK) \to \aC(BG)\\
_{\aC}N_H^G := \aC(BH\leftarrow BH \rightarrow BG) &: \aC(BH) \to \aC(BG)\\
_{\aC}\Phi^H := \aC(BG \leftarrow BG \rightarrow BK) &: \aC(BG) \to \aC(BK).
\end{align*}
\end{notn}

For the example of global equivariant stable homotopy theory,
$\aC(BG)$ is the $\infty$-category $\Sp^G$ of $G$-spectra.  The
functor $\mathrm{res}^{G}_{H}$ is the usual restriction functor from
$G$-spectra to $H$-spectra of (for example) \cite[V.2.1]{MM}.  The
functor $\mathrm{inf}^{G}_{K}$ is the functor $\epsilon^{\#}$
of~\cite[II\S8]{LMS}, left adjoint to the $H$-fixed point functor
$\Sp^{G}\to \Sp^{K}$.  The functor $N_{H}^{G}$ is the
Hill-Hopkins-Ravenel norm functor of~\cite[A.52]{HHR}, and the functor
$\Phi^{H}$ is the geometric fixed point functor of (for example)
\cite[V.4.3]{MM}.  Although it is not (to our knowledge) fully in
the literature and we do not do the required technical work here, this
expands to an example of a global equivariant context.

\begin{prop}\label{prop:sp}
There is a global equivariant context $\underline{\Sp}(-)\colon \Span(\Gpd) \to \Cat_{\infty}$ which sends $BG$ to the $\infty$-category $\Sp^G$ of $G$-spectra and the functors of~\ref{notn:functoriality} to the corresponding functors in genuine equivariant stable homotopy theory.
\end{prop}

As the full subcategory $\Span(\Fin) \subset \Span(\Gpd)$ spanned by
the finite sets is the free symmetric monoidal $\infty$-category on a
point, the restriction of any global equivariant context $\aC$ to
$\Span(\Fin)$ exhibits $\aC(*)$ as a symmetric monoidal
$\infty$-category.  More generally, each of the $\aC(X)$ is endowed
with a natural symmetric monoidal structure, and the functors
of~\ref{notn:functoriality} are naturally symmetric monoidal.   

\begin{defn}\label{defn:global}
Let $\Xi\colon \widetilde{\Sp}\to \Span(\Gpd)$ denote the
coCartesian fibration corresponding to the multiplicative global
equivariant context $\underline{\Sp}$ of Proposition~\ref{prop:sp} and let
$\mathfrak{res}\subset \mathrm{Mor}(\Span(\Gpd))$ be the subset of
spans of the form 
\[
X \xleftarrow{f} Y \xrightarrow{=} Y
\]
where $f$ is
a finite cover of the classifying spaces (up to weak equivalence).  Then a global
commutative ring spectrum with multiplicative deflations is a section of $\Xi$ which is coCartesian
over $\mathfrak{res}$. 
\end{defn}

\begin{rem}
The notion of a global commutative ring spectrum with multiplicative
deflations is closely related to notions appearing in work of
Schwede~\cite{Schwede-Global} (where it might be called an
ultracommutative monoid with multiplicative deflations),
Bachmann-Hoyois~\cite{BachmannHoyois}, and the third
author~\cite{Yuan-Frob}.
\end{rem}

Let $\underline{R}\colon \Span(\Gpd) \to \widetilde{\Sp}$  be a global
commutative ring spectrum with multiplicative deflations. Then we have
$\underline{R}(BG) \in \Sp^G$, as $\Sp^G$ is the fiber of $\Xi$ over
$BG$.  We think of $\underline{R}(BG)$ as the underlying
$G$-equivariant commutative ring spectrum of $\underline{R}$. It
admits natural maps (with reference to \ref{notn:functoriality}):
\begin{align*}
\res^G_H \underline{R}(BG) &\simeq \underline{R}(BH)\\
\mathrm{inf}_K^G\underline{R}(BK) &\to \underline{R}(BG) \\
N_H^G \underline{R}(BH) &\to \underline{R}(BG)\\
\Phi^H \underline{R}(BG) &\to \underline{R}(BK).
\end{align*}
In particular, looking at the last two maps with $H$ and $K$ the trivial group
(using the trivial group as both a subgroup and quotient group of
$G$), $\underline{R}$ has the structure maps needed in the
factorization homology discussion.  The identity
\[
(BG \leftarrow BG \rightarrow *) \circ (* \leftarrow * \rightarrow BG)
= (* \leftarrow * \rightarrow *)
\]
shows that the composition in $\Sp=\aC(*)$
\[
\Phi^{G}N_{e}^{G}\underline{R}(*)\to \Phi^{G}\underline{R}(BG)\to \underline{R}(*)
\]
is equivalent to the identity map $\underline{R}(*)=\underline{R}(*)$,
which is the condition on these maps that we identified above in the
factorization homology discussion and which formed the underlying
condition in the definition of precyclotomic base.

We note that treatment of the equivariance of geometric fixed points
requires consideration of quotient groups for general $G$; for finite
subgroups of $\bT$, the root isomorphisms $\rho_{n}\colon \bT\iso
\bT/C_{n}$ provide a way around this.  For $G=C_{pn}$, the structure
on $\underline{R}$ gives maps in $\Sp^{C_{pn}/C_{p}}$,
\[
\Phi^{C_{p}}\underline{R}(BC_{pn})\to \underline{R}(B(C_{pn}/C_{p}))
\]
which forget under restriction to the map 
\[
\Phi^{C_{p}}\underline{R}(BC_{p})\to \underline{R}(*).
\]
Recalling that $\rho^{*}$ denotes the functor $\Sp^{C_{pn}/C_{p}}\overto{\simeq}
\Sp^{C_{n}}$ coming from the $p$th root isomorphism $\rho \colon
C_{n}\iso C_{pn}/C_{p}$ and the convention $\Phi
=\rho^{*}\Phi^{C_{p}}$, the previous maps take the form 
\[
\Phi \underline{R}(BC_{pn})\to \underline{R}(BC_{n})
\]
in $\Sp^{C_{n}}$, and they are compatible under restriction with the system of finite subgroups $C_{n}$ of $\bT$.

When $\underline{R}$ is a global commutative ring spectrum with
multiplicative deflations, then the above discussion suggests that
there should be a global equivariant context $\underline{\Mod}_{\underline{R}}(-)$ with
\[
\underline{\Mod}_{\underline{R}}(BG) \simeq \Mod_{\underline{R}_G}(\Sp^G),
\]
restriction given by the forgetful functors, the norm functor
given by ${}_{R}N_{H}^{G}$, and the geometric fixed point functor
given by ${}_{R}\Phi^{H}$.  Then, one should be able to define
equivariant factorization homology relative to $\underline{R}$ as
equivariant factorization homology internal to the global
equivariant context $\underline{\Mod}_{\underline{R}}(-)$, at least in
the simplified form discussed above, i.e., in terms of prolonging a
norm functor from finite $G$-sets to all $G$-spaces.

In general, one does not expect naturally occurring global ring spectra
$\underline{R}$ to admit ``multiplicative deflation'' maps
$\Phi^{G}\underline{R} \to \underline{R}$.  For instance, for global
equivariant $K$-theory, we have that $p$ is invertible in
$\Phi^{C_p}\underline{KU}$ but not in $KU$, so it cannot admit such
structure.  Nevertheless, in addition to the sphere spectrum $\bS$, we
have the following example of interest, which we discuss in more
detail in the following section. 

\begin{example}
The complex cobordism ring spectrum $MU$ and its periodic variant
$MUP$ extend to global commutative ring spectra with multiplicative
deflations.  To see this, let $\Vect_{\mathbb{C}}$ denote the
topological category of finite dimensional complex vector spaces and
consider the functor  
\begin{equation*}\label{eqn:ku}
\underline{ku}\colon \Span(\Gpd) \to \Sp
\end{equation*}
which sends $X$ to the underlying spectrum of the symmetric monoidal
topological category $\Fun(X,\Vect_{\mathbb{C}})$ and sends a span
$X \xleftarrow{f} Z\xrightarrow{g} Y$ to the functor given by
restriction along $f$ followed by left Kan extension along $g$.  There
should be a natural global $J$-homomorphism  
\[
\underline{ku} \to \underline{\mathrm{pic}}(\underline{\Sp})
\]
(essentially given by sending $V\in \Fun(X,\Vect_{\mathbb{C}})$ to the
invertible equivariant spectrum $S^V$), and the global Thom spectrum
of this would be a global commutative ring spectrum with
multiplicative deflations lifting the usual global equivariant
structure on $\uMUP$.  One can also expect to obtain $MU$ as the
underlying commutative ring spectrum of a global commutative ring
spectrum with multiplicative deflations using a variant of this
procedure, by taking the global Thom spectrum of $\underline{bu} =
\tau_{\geq 2} \underline{ku}$ instead of $\underline{ku}$.  We denote
this global equivariant spectrum as $\uMUC$.  We explain these
examples carefully in Example~\ref{ex:MU}, but we note here that
$\uMUC$ is not the usual global equivariant structure on $MU$ for
homotopical equivariant cobordism nor for geometric equivariant cobordism.
This can be seen, for example, by looking at $C_{p}$ geometric fixed
points.  The $C_{p}$ geometric fixed points of an equivariant Thom
spectrum is the Thom spectrum on the $C_{p}$ fixed points of the base.
$\uMUC$ as constructed here, homotopical equivariant cobordism,
and geometric equivariant cobordism are equivariant Thom 
spectra for three equivariant structures on $BU$; these three
structures have different $C_{p}$ fixed points. 
\end{example}

\section{Examples}\label{sec:examples}

In this section, we observe that many examples of interest fit into
our framework.  The original example of a cyclotomic base is the
sphere spectrum, where the operation $\Psi$ is the identity map on the
point-set level.  Work of Bhatt-Morrow-Scholze~\cite[\S 11.1]{BMS2} on
relative $TC$ suggests that the spherical monoid ring
$\bS[t]=\Sigma^{\infty}_{+}\bN$ should have such a structure.  The
other examples we discuss were not previously known.

\numberwithin{equation}{subsection}
\subsection{Example: $\bS[G]$ and $\bS[\bN]$}\quad\par\medskip

Let $G$ be an abelian group.  We consider the opposite constant
$\bT$-Mackey functor on $G$, which we denote as $\uG$: specifically,
for a closed subgroup $H<\bT$, $\uG(\bT/H)=G$, and for a pair of
closed subgroups $K<H<\bT$, the transfer $\uG(\bT/K)\to \uG(\bT/H)$ is
the identity, and the restriction $\uG(\bT/H)\to \uG(\bT/K)$ is
multiplication by $\chi(H/K)$, that is, multiplication by the index if
finite and zero if not.  Associated to $\uG$ is the (genuine)
$\bT$-equivariant Eilenberg-Mac\,Lane spectrum $H\uG$, and for any
model, the derived zeroth space $\Omega^{\infty}H\uG$ has a canonical
genuine $\bT$-equivariant $E_{\infty}$ space structure, or more concisely, an
$E^{\bT}_{\infty}$ space structure.  Fixing a model,
$\Omega^{\infty}H\uG$, the underlying non-equivariant grouplike
$E_{\infty}$ space is equivalent to $G$, and we regard
$\Omega^{\infty}H\uG$ as a $\bT$-equivariant version of $G$.

We see from the concrete description of $\uG$ that the $\bT$-Mackey
functor $\upi_{0}(\rho^{*}((H\uG)^{C_{p}}))$ is again $\uG$.  We then get
a weak equivalence 
\[
\rho^{*}((H\uG)^{C_{p}})\to H\uG.
\]
Applying the derived zeroth space functor, we get an
$E_{\infty}^{\bT}$ space map
\begin{equation}\label{eq:HuGr}
\rho^{*}((\Omega^{\infty}H\uG)^{C_{p}})\to \Omega^{\infty}H\uG
\end{equation}
that is a weak equivalence. 

Let $\uR=\bS[\Omega^{\infty}H\uG]$, the equivariant spherical group ring
on $\Omega^{\infty}H\uG$: as a genuine $\bT$-spectrum, $\uR$ is the
unreduced suspension spectrum on $\Omega^{\infty}H\uG$, 
\[
\uR=\Sigma^{\infty}_{+}(\Omega^{\infty}H\uG),
\]
but the $E^{\bT}_{\infty}$ space structure on $\Omega^{\infty}H\uG$
endows $\uR$ with the structure of an $E^{\bT}_{\infty}$ ring spectrum.
Using cofibrant replacement, we can find a weakly equivalent
$E^{\bT}_{\infty}$ ring spectrum which is a $\bT$-equivariant
commutative ring orthogonal spectrum.

\begin{notn}
Let $\bS_{\bT}[G]$ denote a $\bT$-equivariant commutative ring
orthogonal spectrum model of $\bS[\Omega^{\infty}H\uG]$.
\end{notn}

The point-set multiplicative geometric fixed point functor satisfies 
\[
\Phi^{C_{p}}\uR=\Phi^{C_{p}}(\Sigma^{\infty}_{+}(\Omega^{\infty}H\uG))
=\Sigma^{\infty}_{+}((\Omega^{\infty}H\uG)^{C_{p}}),
\]
and the map~\eqref{eq:HuGr} gives a map of $E^{\bT}_{\infty}$ ring
spectra
\[
r\colon \Phi \uR\to \uR,
\]
which we use as the cyclotomic structure map for $\cR$, giving a
cyclotomic structure map for $\bS_{\bT}[G]$, which is a map of
$\bT$-equivariant commutative ring orthogonal spectra.

\begin{thm}
$\bS_{\bT}[G]\simeq \bS[\Omega^{\infty}H\uG]$ with the cyclotomic structure above
admits the structure of a cyclotomic base.
\end{thm}

\begin{proof}
We need to show that the operation 
\[
\Psi \colon \bS_{\bT}[G]\to \bS_{\bT}[G],
\]
is the identity in the homotopy category of commutative ring
orthogonal spectra.  Recall that this operation is the composite of
the norm diagonal followed by multiplication 
\[
\bS_{\bT}[G]\to \Phi^{C_{p}}N_{e}^{C_{p}}\bS_{\bT}[G]\to \Phi^{C_{p}}\bS_{\bT}[G]
\]
and (the non-equivariant map underlying) the cyclotomic structure map
$\Phi \bS_{\bT}[G]\to \bS_{\bT}[G]$.  We can understand the displayed map in terms of the
spherical group ring $\bS[\Omega^{\infty}H\uG]$: it is just the induced
map on $\bS[-]$ of the transfer
\[
\Omega^{\infty}H\uG\to (\Omega^{\infty}H\uG)^{C_{p}},
\]
which is just $\Omega^{\infty}$ applied to the transfer
\[
H\uG\to (H\uG)^{C_{p}},
\]
which by definition of $\uG$ is homotopic to the identity, under the
identification $(H\uG)^{C_{p}}\simeq H\uG$ defining the cyclotomic
structure map.
\end{proof}

Technically the definition of cyclotomic base requires a choice of
point-set homotopy. When the model $\bS_{\bT}[G]$ is cofibrant and
fibrant in the model structure on commutative ring cyclotomic spectra
of~\cite[Thm.~D]{BM-cycl23}, the underlying (non-equivariant) commutative
ring orthogonal spectrum is cofibrant and fibrant in the model
structure of \cite[III.4.1]{BM-cycl23}
(by~\cite[III.1.3]{BM-cycl23}) and the argument above
constructs a contractible space of
homotopies.  Any choices produce weakly equivalent cyclotomic bases,
as per Proposition~\ref{prop:disk} in Section~\ref{sec:relcyc}.

We can calculate $TC$ of this example (up to $p$-completion), using
the original method of 
B\"okstedt-Hsiang-Madsen~\cite{BHM}.  The exposition of Madsen
in~\cite[4.4.3]{Madsen-Traces} simplifies the argument of~\cite{BHM}
and it is clear that it applies to any cyclotomic spectrum $\cX$ where
the map from the categorical to geometric fixed points $\uX^{C_{p}}\to
\Phi^{C_{p}}\uX$ is $\bT/C_{p}$-equivariantly split.
(Nikolaus-Scholze~\cite[IV.3.4]{NikolausScholze} make an analogous
observation in their setting, in terms of the notion of a Frobenius
lift.)  Madsen's splitting always
happens for suspension spectra.  In this case, the statement
specializes to the following.

\begin{prop}For $\bS_{\bT}[G]$ as above, $TC_{\cyc}(\bS_{\bT}[G])\phat$ fits
into a homotopy fiber square
\[
\xymatrix@-1pc{
TC_{\cyc}(\bS_{\bT}[G])\phat\ar[r]\ar[d] 
&((\Sigma \bS[G])_{h\bT})\phat \ar[d]^{\mathrm{tr}}\\
\bS[G]\phat\ar[r]_{{[p]}-\id} 
&\bS[G]^{\wedge}_p
}
\]
where $[p]$ denotes the map on $\bS[G]$ induced by the multiplication
by $p$ map (the $p$-power map) on the abelian group $G$.
\end{prop}

The above discussion in terms of groups can be
extended to monoids that inject into their group completion; we treat
in detail only the case of the natural numbers $0,1,2,3,\dotsc$.
Starting with $G=\bZ$, we get an $E^{\bT}_{\infty}$ space
$\Omega^{\infty}H\ubZ$. Non-equivariantly, the
components of $\Omega^{\infty}H\ubZ$ are canonically in
one-to-one correspondence with the integers; let
$\uOZN$ be the subspace of
those components corresponding to the natural numbers
$0,1,2,\dotsc$. The $\bT$-action and $E^{\bT}_{\infty}$ structure on 
$\Omega^{\infty}H\ubZ$ restrict to
$\uOZN$.  The spherical monoid ring
$\bS[\uOZN]$ then obtains an $E^{\bT}_{\infty}$ ring
structure.  This is the appropriate equivariant version of the
$E_{\infty}$ ring spectrum $\bS[t]$ for the cyclotomic structure
studied by Bhatt-Morrow-Scholze in~\cite[11.1]{BMS2}.

\begin{notn}
Write $\bS_{\bT}[t]$ for a $\bT$-equivariant commutative ring orthogonal spectrum model for
the $E^{\bT}_{\infty}$ ring spectrum $\bS[\uOZN]$.
\end{notn}

The map~\eqref{eq:HuGr} restricts to a weak equivalence of
$E^{\bT}_{\infty}$ spaces
\[
\rho^{*}((\uOZN)^{C_{p}})\to \uOZN.
\]
We use this to give $\bS_{\bT}[t]$ the structure of a commutative ring cyclotomic spectrum.

\begin{rem}
It is easy to see that the cyclotomic structure above is the (classic)
cyclotomic structure on $\bS[t]$ described by Bhatt-Morrow-Scholze
in~\cite[11.5]{BMS2}: the composite 
\[
\bS[t]:=\bS[\bN]\simeq \bS_{\bT}[t]\overfrom{\simeq}\Phi (\bS_{\bT}[t])
\to \rho^{*}((\bS_{\bT}[t])^{tC_{p}})\simeq \rho^{*}((\bS[\bN])^{tC_{p}})
\]
is induced by the multiplication by $p$ map $\bN\to \bN$.  (In the
display ``$\simeq$'' denotes $\bT$-equivariant Borel equivalence and
$\bS[\bN]$ has the trivial $\bT$-action.)
\end{rem}

Bhatt-Morrow-Scholze~\cite[\S11.1]{BMS2} study cyclotomic structures
relative to $\bS[t]$; our theory also gives such structures:

\begin{thm}
$\bS_{\bT}[t]\simeq \bS[\uOZN]$ with the cyclotomic
structure above admits the structure of a cyclotomic
base, i.e., $\Psi \colon \bS_{\bT}[t]\to
\bS_{\bT}[t]$ is the identity in the homotopy category of $E^{\bT}_{\infty}$ ring
spectra.
\end{thm}

The proof follows from the corresponding proof above for $\ubZ$, noting
that the transfer 
\[
\Omega^{\infty}H\ubZ\to (\Omega^{\infty}H\ubZ)^{C_{p}}
\]
restricts to
\[
\uOZN\to (\uOZN)^{C_{p}}.
\]

\subsection{Example: The multiplicative Borel precyclotomic structure}\qquad \par\medskip

Given a (non-equivariant) cofibrant commutative ring orthogonal
spectrum $R$, we can apply the functor $\epsilon$ of
Notation~\ref{notn:epsilon} to get a $\bT$-equivariant commutative
ring orthogonal spectrum $\epsilon R$; we then tensor in the category
of $\bT$-equivariant commutative ring orthogonal spectra with the
$\bT$-space $E\bT$ to form a new $\bT$-equivariant commutative ring
orthogonal spectrum $\epsilon R\otimes E\bT$.  We have the following
identification of its geometric fixed points.

\begin{prop}\label{prop:mbpid}
Let $R$ be a cofibrant non-equivariant commutative ring orthogonal spectrum.
There exists a natural isomorphism of $\bT$-equivariant commutative
ring orthogonal spectra
\[
\Phi(\epsilon R\otimes E\bT)\iso \epsilon R\otimes \rho^{*}(E\bT/C_{p}).
\]
\end{prop}

\begin{proof}
Working with the commutative ring orthogonal spectrum $R\otimes X$ for
a space $X$, the diagonal map $N_{e}^{\bT/C_{p}}\to
\Phi^{C_{p}}N_{e}^{\bT}$ 
of~\cite[4.6]{ABGHLM} induces a $\bT$-equivariant map 
\[
\delta \colon \epsilon R\otimes \rho^{*}(X\times \bT/C_{p})\to
\Phi (\epsilon R\otimes (X\times \bT)),
\]
natural in the commutative ring orthogonal spectrum $R$ and the space
$X$, which is an isomorphism when $R$ is cofibrant and $X$ is a cell
complex. Moreover, since $R$ is cofibrant as a commutative ring
orthogonal spectrum, $\epsilon R$ is cofibrant as a $\bT$-equivariant commutative
ring orthogonal spectrum, and the above generalized diagonal map is an
isomorphism.  By the universal property of
$N_{e}^{\bT/C_{p}}$ (the $\bT/C_{p}$ analogue of
Proposition~\ref{prop:tensor}), any equivariant map
\[
\epsilon R\otimes \rho^{*}(X\times \bT/C_{p})\to
\Phi (\epsilon R\otimes (X\times \bT))
\]
is uniquely characterized by the composite map of non-equivariant
commutative ring orthogonal spectra
\[
R\otimes X\to i(\epsilon R\otimes \rho^{*}(X\times \bT/C_{p}))\to
i(\Phi (\epsilon R\otimes (X\times \bT))).
\]
Now consider the diagram 
\[
\xymatrix{%
\epsilon R\otimes \rho^{*}(\bT\times \bT/C_{p})
\ar[r]^{\delta}\ar[d]
&\Phi (\epsilon R\otimes (\bT\times \bT))\ar[d]
\\
\epsilon R\otimes \rho^{*}(\bT/C_{p})\ar[r]_{\delta}
&\Phi (\epsilon R\otimes \bT)
}
\]
where the vertical maps are induced by the action maps $\bT\times
\bT/C_{p}\to \bT/C_{p}$ and $\bT\times \bT\to \bT$.  Looking at the
underlying non-equivariant diagram and precomposing with the
non-equivariant map  
\[
R\otimes \bT \to i(\epsilon R\otimes \rho^{*}(\bT\times \bT/C_{p}))
\]
both composite maps
\[
R\otimes \bT\to i(\Phi (\epsilon R\otimes \bT))\iso R\otimes \bT
\]
are the identity map, so the diagram commutes.
Applying these observations to $X=E\bT$ defined in terms of the bar
construction
\[
E\bT=|B\subdot (*,\bT,\bT)|=|\bT\times \dotsb \times \bT\times \bT|,
\]
we see that $\delta$ then induces a natural isomorphism of simplicial
$\bT/C_{p}$-equivariant commutative ring orthogonal spectra 
\[
\epsilon R\otimes \rho^{*}B\subdot(*,\bT,\bT)/C_{p}\to 
\Phi(\epsilon R\otimes B\subdot(*,\bT,\bT)).
\]
Taking geometric realization, we get the natural isomorphism of the
statement. 
\end{proof}

We observe that the $\bT$-space $\rho^{*}(E\bT/C_{p})$ is a free $\bT$-CW
complex non-equivariantly homotopy equivalent to $BC_{p}$; we choose and
fix a $\bT$-equivariant homotopy equivalence $\rho^{*}(E\bT/C_{p})\to
BC_{p}\times E\bT$. 

\begin{defn}
The \term{multiplicative Borel precyclotomic structure} on a cofibrant
(non-equivariant) commutative ring orthogonal spectrum $R$ consists of
the genuine $\bT$-equivariant commutative ring orthogonal spectrum
$\epsilon R\otimes E\bT$, with the precyclotomic structure map
\[
\Phi(\epsilon R\otimes E\bT)\iso \epsilon R\otimes
\rho^{*}(E\bT/C_{p})\to \epsilon R\otimes (BC_{p}\times E\bT)
\to \epsilon R\otimes (E\bT)
\]
induced by the isomorphism of the proposition above, the homotopy
equivalence chosen above, and the collapse
map $BC_{p}\to *$.  
\end{defn}

\begin{prop}
For $R$ a cofibrant (non-equivariant) commutative ring orthogonal
spectrum, the multiplicative Borel precyclotomic spectrum
$\uR=\epsilon R\otimes E\bT$ admits the natural structure of a
precyclotomic base. 
\end{prop}

\begin{proof}
It is clear from the construction of the isomorphism of
Proposition~\ref{prop:mbpid} that the map
\[
R \to \Phi^{C_{p}}N_{e}^{C_{p}} R\to \Phi^{C_{p}}\uR
\]
is the map $\epsilon R\otimes E\bT\to \epsilon R\otimes (E\bT/C_{p})$
induced by the quotient map $E\bT\to E\bT/C_{p}$.  The composite of
this and the precyclotomic structure map is then the self-map of
$\epsilon R\otimes E\bT$ induced by the map
\[
E\bT\to E\bT/C_{p}\to BC_{p}\times E\bT\to E\bT,
\]
which is evidently (non-equivariantly) homotopic to the identity. 
\end{proof}

We have a variant of this precyclotomic base structure using $\epsilon R$
without first tensoring with $E\bT$: the multiplicative tom Dieck
splitting of~\cite[Thm.~A]{BM-cycl23} identifies $\Phi \epsilon R$
(up to weak equivalence) as $R\sma (R\otimes BC_{p})$ with the
multiplicative transfer the map 
\[
R\iso R\otimes * \to R\otimes BC_{p}\to R\sma (R\otimes BC_{p})
\]
induced by the inclusion of the base point in $BC_{p}$.  If we give
$\epsilon R$ the precyclotomic structure
\[
\Phi \epsilon R \simeq R\sma (R\otimes BC_{p})\to R
\]
which is the identity on the $R$ factor and the map 
\[
R\otimes BC_{p}\to R\otimes *\iso R
\]
on the other factor, then $\Psi \colon R\to R$ is homotopic to the identity.

As a particular example, we get two $\bT$-equivariant versions of $H\bZ$
with precyclotomic base structures.  This does not contradict Hesselholt's
observation~\cite[7.1]{ABGHLM} on the non-existence of a cyclotomic
structure on relative $THH^{H\bZ}$, because the structures induced on
$THH^{H\bZ}$ are precyclotomic and not cyclotomic.

\subsection{Example: $H\uFp$}\label{ex:Fp}\quad\par\medskip

Let $H\uFp$ be a cofibrant model for the $\bT$-equivariant commutative ring spectrum
specified by the Eilenberg-Mac\,Lane spectrum on the constant Mackey
functor on $\bF_{p}$ (with constant restriction maps and zero transfer
maps).  In this case, the geometric fixed points $\Phi^{C_{p}}H\uFp$ is
the connective cover of the Tate fixed point spectrum
$H\uFp^{tC_{p}}$.  In particular, it is a connective
$\bT/C_{p}$-equivariant commutative ring orthogonal spectrum with
$\underline\pi_{0}\iso \uFp$ as a Green functor.  It follows that there
exists a unique map in the homotopy category of $\bT$-equivariant
commutative ring orthogonal spectra $\Phi H\uFp\to H\uFp$ and a unique
homotopy type in commutative ring precyclotomic spectra with this map
as the structure map.  The cyclotomic power operation $\Psi$ is a map of
commutative ring orthogonal spectra $H\bF_{p}\to H\bF_{p}$ and is
therefore homotopic to the identity, with a contractible space of
choices for the homotopy.

In this example, it is easy to compute relative $TC$.  In the
statement, $D(S^{1}_{+})$ denotes the Spanier-Whitehead dual of
$S^{1}_{+}$ with the commutative ring spectrum structure induced by the
diagonal map on $S^{1}$.

\begin{prop}
There exists a weak equivalence of commutative ring spectra between
$TC^{H\underline{\bF}_p}(H\underline{\bF}_p)$ and the homotopy
equalizer of the identity with itself on $H\bF_{p}$, 
\[
TC^{H\underline{\bF}_p}(H\underline{\bF}_p) \simeq \HoEq(\id,\id\colon 
H\bF_p\to H\bF_p)\simeq H\bF_{p}\sma D(S^{1}_{+}).
\]
\end{prop}

\begin{proof}
We can compute $TC$ (as a commutative ring spectrum) as the homotopy
equalizer of the maps $R$ and $F$ on 
\[
TF^{H\underline{\bF}_p}(H\underline{\bF}_p) := \holim_{n} H\underline{\bF}_p^{C_{p^n}}
\]
where the limit is taken along the inclusion of fixed point maps $F$
(see for example~\cite[\S 2.5]{Madsen-Traces} for this fact and the definition
of $R$ and $F$).
The maps $F$ are maps of commutative ring spectra $H\bF_{p}\to
H\bF_{p}$ and so are homotopic to the identity (as maps of commutative
ring spectra) in an essentially unique way.  This gives us a weak
equivalence of commutative ring spectra
$TF^{H\underline{\bF}_p}(H\underline{\bF}_p)\simeq H\bF_{p}$.  The
self-maps $R$ and $F$ of $TF^{H\underline{\bF}_p}(H\underline{\bF}_p)$
are maps of commutative ring spectra and so also homotopic to the identity
in an essentially unique way.
\end{proof}

Although the underlying non-equivariant spectra of both
$THH^{H\uFp}(HA)$ and $THH^{\epsilon H\bF_{p}}(HA)$ are equivalent to
$HH^{\bF_{p}}(A)$ (for any $\bF_{p}$-algebra $A$), we note that these
are quite different equivariant spectra.  For example,
non-equivariantly 
\[
\bL\Phi
(THH^{H\uFp}(H\bF_{p}))\simeq \Phi(H\uFp)\simeq \tau_{\geq 0}H\uFp^{tC_{p}},
\]
which has as its homotopy groups a single copy of $\bF_{p}$ in each
non-negative degree, whereas 
\[
\bL\Phi (THH^{\epsilon
H\bF_{p}}(H\bF_{p}))\simeq H\bF_{p}\sma (H\bF_{p}\otimes BC_{p})
\]
(see Warning~\ref{warn:epsilon}),
which has as its homotopy groups a free module over the dual Steenrod
algebra.  For more on $\epsilon H\bF_{p}$ (and a precyclotomic base
structure on it), see the example on multiplicative Borel
precyclotomic structures above.

\subsection{Example: connective homotopical and periodic equivariant
complex\break cobordism}\label{ex:MU}\qquad \par\medskip

Given an equivariant map of $E^{G}_{\infty}$ spaces $\alpha \colon X\to
\uBUZ$, we get a $G$-equivariant $E^{G}_{\infty}$ ring Thom spectrum
$M(\alpha)$ \cite[X.3]{LMS}, which is usually denoted $MX$ by abuse of notation.  $M(\mathrm{id})$ is
the equivariant Thom spectrum $\uMUP$, periodic equivariant complex
cobordism. This construction is given by Lewis in Chapter~X
of~\cite{LMS} though not named per se; non-equivariantly, it first
appears in work of 
Snaith~\cite{Snaith-MUP}, but with the wrong $E_{\infty}$ multiplicative
structure~\cite{HahnYuan-Snaith}.
As a non-equivariant stable homotopy type, 
\[
MUP\simeq MU[u^{\pm1}]\simeq \bigvee_{n\in \bZ}\Sigma^{2n}MU
\]
(where $|u| = 2$).

The inclusion of $\uBU$ in $\uBUZ$ as the zero component
(non-equivariantly: on fixed points it includes as the components of
virtual dimension zero in the representation ring) produces the
equivariant Thom spectrum $\uMU$, homotopical equivariant complex
cobordism.  In this section, we introduce a new $G$-equivariant
$E^{G}_{\infty}$ ring Thom spectrum $\uMUC$ which we call connective
homotopical equivariant complex cobordism, which is the Thom spectrum of
a map of $E^{G}_{\infty}$ spaces $\uBUC\to \uBUZ$, where on each fixed
point subspace $\uBUC$ includes as the zero component.  We emphasize
that despite similar notation, $\uMUC$ is more related to 
homotopical equivariant cobordism than equivariant geometric cobordism. (The Thom
spectrum model of the geometric theory for equivariant unoriented
cobordism (for groups satisfying Wasserman's condition) is now
generally denoted 
$mO$~\cite{Schwede-Global}, but in older literature was denoted
$mo$~\cite[\S XV]{May-Alaska}.) 

The equivariant space $\uBUC$ and map $\uBUC\to \uBUZ$, without further
structure, are easily constructed by Elmendorf's theorem: working with
the orbit diagram for $\uBUZ$, we take the subdiagram that at the
orbit $G/H$ consists of the zero component.  Non-equivariantly, the
zero component is a product of copies of $BU$ indexed on the
irreducible representations of $H$.  The fact that these subspaces
should be closed under transfers suggests that there should exist an
$E^{G}_{\infty}$ model.  To construct this, we note that as an
$E^{G}_{\infty}$ space 
\[
\uBUZ
\simeq \Omega^{\infty}\tau_{\geq 0}\uKU,
\]
where $\tau_{\geq 0}\uKU$ is the genuine $G$-equivariant spectrum
characterized (up to weak equivalence) by $\pi^{H}_{q}=0$ for $q<0$
and having the extra structure of a map to $\uKU$ which is an
isomorphism of ($\bZ$-graded) homotopy groups in non-negative degrees.
There is an essentially unique map of genuine $G$-equivariant spectra
from $\tau_{\geq 0}\uKU$ to $H\uRep$ that is the identity on
$\upi_{0}$, where $\uRep$ is the complex representation ring Mackey
functor.  We then get a map of $E^{G}_{\infty}$ spaces, 
\[
\uBUZ\to \Omega^{\infty}H\uRep
\]
and we obtain an $E^{G}_{\infty}$ space model of $\uBUC$ as the
homotopy fiber.

\begin{notn}
Let $\uMUC$ be the $E^{G}_{\infty}$ Thom spectrum associated to the
$E^{G}_{\infty}$ space map $\uBUC\to \uBUZ$.
\end{notn}

The geometric fixed point functor commutes with the Thom spectrum
functor in the following sense. For clarity, we add a subscript to
denote the group of equivariance, as in $\uBUZ_{G}$ (for the 
$E^{G}_\infty$ space $\uBUZ$).  For a subgroup
$H<G$, let $WH$ denote its Weyl group.  We have a map of
$E^{WH}_{\infty}$ spaces 
\begin{equation}\label{eq:fpmap}
\phi^{H}\colon (\uBUZ_{G})^{H}\to \uBUZ_{WH}
\end{equation}
induced by replacing an $H$-equivariant complex vector space with its
$H$-fixed point subspace.  Starting with a $G$-equivariant map 
\[
\alpha \colon X\to \uBUZ
\]
defining a genuine $G$-equivariant Thom spectrum $MX$, taking
$H$-fixed points, we get a map of $WH$-spaces
\[
\alpha^{H}\colon X^{H}\to (\uBUZ)^{H}\simeq BU_{WH}\times R(H).
\]
Let \[
\phi^{H}\alpha^{H}\colon X^H \to \uBUZ_{WH}
\]
be the composite of
$\alpha^{H}$ with $\phi^{H}$.  
Then the geometric fixed points $\Phi^H MX$
is weakly equivalent to the genuine $WH$-equivariant Thom spectrum
$M(X^{H})$ for the map $\phi^{H}\alpha^{H}$.
When
$\alpha$ is an $E^{G}_{\infty}$ map, $\alpha^{H}$ is an
$E^{WH}_{\infty}$ map, $M(X^{H})=M(\phi^{H}\alpha^{H})$ is an
$E^{WH}_{\infty}$ ring spectrum, and we get a weak equivalence of
$E^{WH}_{\infty}$ ring spectra $\Phi^H MX \simeq M(X^{H})$.

We now discuss precyclotomic structures and set $G=\bT$.  In the case
of $\uMUP$, the existence of such a structure was originally observed
by Brun~\cite[8.3]{Brun-MUP}.  In our notation, the map~\eqref{eq:fpmap} for
$H=C_{p}$ induces a map of $E^{\bT/C_{p}}_{\infty}$ ring spectra
\[
\Phi^{C_{p}}(\uMUP_{\bT})\to \uMUP_{\bT/C_{p}}.
\]
Applying the functor $\rho^{*}$, we then get a map of
$E^{\bT}_{\infty}$ ring spectra
\[
r\colon \Phi \uMUP\to \uMUP,
\]
which we take as the precyclotomic structure map.  We note that this
map is not a weak equivalence as the domain and codomain are not even
abstractly weakly equivalent; see~\cite[4.10]{Sinha-TBordism}.

In the case of $\uMUC$, the diagram of genuine $\bT/C_{p}$-equivariant spectra
\[
\xymatrix{%
(\tau_{\geq 0}KU_{\bT})^{C_{p}}\ar[d]
\ar[r]
&H\uRep^{C_{p}}\ar[d]\\
\tau_{\geq 0}KU_{\bT/C_{p}}\ar[r]
&H\uRep
}
\]
commutes up to homotopy where the lefthand vertical map is the spectrum-level
analogue of~\eqref{eq:fpmap} and the righthand vertical map is induced
by the map of Mackey functors that at level $(\bT/C_{p})/(K/C_{p})$ is
the unique (abelian group) homomorphism
$R(K)\to R(K/C_{p})$ that sends a $C_{p}$-trivial irreducible
$K$-representation to the corresponding
$K/C_{p}$-representation and sends $C_{p}$-non-trivial irreducible
$K$-representations to $0$.   The map $r$ therefore lifts up to homotopy to
a map of $E^{\bT}_{\infty}$ ring spectra
\[
r\colon \Phi \uMUC\to \uMUC,
\]
which we use for the precyclotomic structure.  Again this map is not
a weak equivalence: non-equivariantly, the left-hand side is equivalent to 
$MU^{(p)}\simeq MU\sma
(BU_{+})^{(p-1)}$ while the right-hand side is $MU$.

We observe that the trick above does not work to define a
precyclotomic structure on $\uMU$.  Non-equivariantly, $MU$ is
the Thom spectrum of $BU\times \{0\}\to BU\times \bZ$ and
$\Phi^{C_{p}}MU$ is the Thom spectrum of 
\[
BU\times I(R(C_{p}))\to BU\times R(C_{p})\to BU\times \bZ.
\]
The map $V\mapsto V^{H}$ does not send the augmentation ideal
$I(R(C_{p}))$ to $0\in \bZ$, so the map $r\colon \Phi \uMUP\to \uMUP$ does
not restrict to a map $\Phi \uMU\to \uMU$.

\begin{thm}\label{thm:mupcycbase}
Both $\cMUP$ and $\cMUC$ admit the structure of a precyclotomic base.
\end{thm}

\begin{proof}
Let $\cR$ be a commutative ring precyclotomic spectrum model for $\cMUP$ or
$\cMUC$ and let $X=\uBUZ$ or $\uBUC$, respectively.  The first part
of the operation $\Psi$, 
\[
R\to \Phi^{C_{p}}N_{e}^{C_{p}}R\to \Phi^{C_{p}}\uR
\]
is the map of Thom spectra induced by the transfer $X\to
X^{H}$~\cite[IV.1.4]{BM-cycl23}, and the second part is $r$, which is
the map of Thom spectra induced by the map $X^{H}\to X$ described
above.  Both maps are $\Omega^{\infty}$ of spectrum-level maps, and
the composite on the spectrum level is the identity in the stable
category. 
\end{proof}

\subsection{Non-example: $\uKU$}\label{ex:Kth}\quad\par\medskip

Although the geometric construction of the precyclotomic structure map
on $\uMUP$ suggests that the equivariant complex $K$-theory spectrum
$\uKU$ might provide another example of a precyclotomic base, in fact,
it cannot even be a ring precyclotomic spectrum.  By~\cite[3.1]{Greenlees-forms}, the integer
$p$ is a unit in  $\pi_{0}(\Phi^{C_p}\uKU)$, and so there are no maps
of ring spectra $\Phi \uKU \to \uKU$.

\numberwithin{equation}{section}

\section{Precyclotomic spectra and Nikolaus-Scholze $TC(-;p)$}
\label{sec:phiinfty}

Nikolaus-Scholze~\cite{NikolausScholze} describes a shortcut for
calculating $TC_{\cyc}$ of connective cyclotomic spectra.  The purpose of
this section is to convert a connective precyclotomic spectrum to a
cyclotomic spectrum with equivalent $TC_{\cyc}$, which can then be calculated
by the Nikolaus-Scholze formula.  This amounts to describing
concretely the restriction to connective objects of the right adjoint
to the inclusion of the homotopy category of cyclotomic spectra in the homotopy
category of precyclotomic spectra.  

The idea for the functor is to take the homotopy inverse limit of
iterates of $\Phi$.  This is complicated by the fact that to preserve
weak equivalences, we need both the left derived functor of $\Phi$ and
the right derived functor of inverse limit; implicitly this involves
both cofibrant and fibrant approximation.  We use $L$ and $R$ for
cofibrant and fibrant approximation functors, respectively, in the
category of precyclotomic spectra.  In our main construction that
follows, the fibrant approximation functor $R$ is only applied to
cofibrant objects, and in the case of a commutative ring precyclotomic
spectrum, we have a variant construction using cofibrant and fibrant
approximation in the category of commutative ring precyclotomic
spectra (q.v.~\cite[Thm.~D]{BM-cycl23}).  The construction is a homotopy
inverse limit in the category of precyclotomic spectra: for fibrant
objects this may be
constructed as the spacewise homotopy inverse limit (either using the 
Bousfield-Kan cobar construction or the mapping microscope
construction) in orthogonal spectra, given the 
evident precyclotomic structure map.

\begin{cons}
For a precyclotomic spectrum $\cX$ with structure map $r$, we note that
$\Phi \uX$ has the canonical structure of a precyclotomic spectrum
(that we denote $\Phi \cX$) with
structure map $\Phi r$, and that the map $r\colon \Phi \cX\to \cX$ is a
map of precyclotomic spectra.  Writing $\Phi^{n}$ for the $n$th
iterate of $\Phi$, we get an inverse system of precyclotomic spectra
$\Phi^{n}\cX$. Let $\Phi^{\infty}\cX$ be the homotopy inverse limit of the
inverse system of precyclotomic spectra $R\Phi^{n}L \cX$.
\end{cons}

By construction, $\Phi^{\infty}$ preserves weak equivalences of precyclotomic
spectra. Moreover, $\Phi^{\infty} \cX$ comes with a point-set map of
precyclotomic spectra $\Phi^{\infty} \cX\to RL\cX$ and a map in the homotopy
category $\Phi^{\infty} \cX\to \cX$, induced by the identity
$R\Phi^{0}L\cX=RL\cX$. These maps are weak equivalences of precyclotomic
spectra when $\cX$ is cyclotomic, and in particular $\Phi^{\infty} \cX$ is itself
cyclotomic in this case.  We also have the following observation:

\begin{prop}
Let $\cX$ be a precyclotomic spectrum that is
$\aF_{p}$-connective (i.e., $\uX^{C_{p^{n}}}$ is connective for all
$n\geq 0$).  Then $\Phi^{\infty} \cX$ is a cyclotomic spectrum. 
\end{prop}

\begin{proof}
Let $\cX_{n}=R\Phi^{n}L\cX$.  It suffices to show that the natural map
\[
\Phi L(\holim_{n} R\cX_{n})\to \holim_{n}(R\Phi LR\cX_{n})
\]
is an $\aF_{p}$-equivalence.  
Looking at the norm cofiber sequence inside
and outside the $\holim$, and taking $C_{p^{k-1}}$-fixed points, we
get a map of cofiber sequences of non-equivariant spectra
\[
\xymatrix@C-1.25pc@R-.5pc{%
(R((\holim_{n} R\cX_{n})\sma E\bT))^{C_{p^{k}}}\ar[r]\ar[d]
&(R\holim_{n} R\cX_{n})^{C_{p^{k}}}\ar[r]\ar[d]
&(R^{2}\Phi L\holim_{n} R\cX_{n})^{C_{p^{k-1}}}\ar[d]\\
R\holim_{n} (R(\cX_{n}\sma E\bT))^{C_{p^{k}}}\ar[r]
&R\holim_{n} (R\cX_{n})^{C_{p^{k}}}\ar[r]
&R\holim_{n} (R\Phi LR\cX_{n})^{C_{p^{k-1}}}
}
\]
and it suffices to observe that the first two vertical maps are weak
equivalences.  The second vertical map is a weak equivalence without
hypotheses on $\cX_{n}$ and the first is a weak equivalence when the
underlying non-equivariant spectra of $\cX_{n}$ are connective.  (To see
this, let $E\bT_{q}$ denote the $\bT$-equivariant $q$-skeleton of $E\bT$.
Because $E\bT_{q}$ is a finite $\bT$-spectrum, smashing with it commutes
up to weak equivalence with homotopy limits.  Since both $((\holim_{n}
\cX_{n})\sma E\bT/E\bT_{q})^{C_{p^{k}}}$ and $\holim_{n} (\cX_{n}\sma
E\bT_q)^{C_{p^{k}}}$ are at least $q$-connected, the first map is a weak
equivalence.)
\end{proof}

\begin{prop}
Let $\cT$ be a cyclotomic spectrum and $\cX$ a precyclotomic spectrum.
Then the map $\Phi^{\infty} \cX\to \cX$ in the homotopy category of precyclotomic
spectra induces a weak equivalence 
\[
\bR F_{\cyc}(\cT,\Phi^{\infty} \cX)\to \bR F_{\cyc}(\cT,\cX)
\]
where $\bR F_{\cyc}$ denotes the derived mapping spectrum of precyclotomic maps.
\end{prop}

\begin{proof}
We can assume without loss of generality that $\cT$ is cofibrant; then
since $\cT$ is cyclotomic, $\Phi \cT$ is also cyclotomic, and $r\colon
\Phi \cT\to \cT$ is a weak equivalence.   Because $\Phi$, viewed as an
endofunctor on precyclotomic spectra, is spectrally enriched, it
follows that the map $r\colon \bL\Phi \cX\to \cX$ induces an isomorphism in the
stable category  
\[
\bR F_{\cyc}(\cT,\bL\Phi \cX)\to \bR F_{\cyc}(\cT,\cX),
\]
with the composite
\[
\bR F_{\cyc}(\cT,\cX)\overto{\Phi}\bR F_{\cyc}(\Phi \cT,\bL\Phi \cX)
\overto{(r^{-1})^*} \bR F_{\cyc}(\cT,\bL\Phi \cX)
\]
giving the inverse isomorphism. (This assertion is the tautological
observation that for any precyclotomic map $f\colon \cT\to \cX$, the diagram
\[
\xymatrix@-1pc{%
\Phi \cT\ar[r]^{\Phi f}\ar[d]_{r}&\Phi \cX\ar[d]^{r} \\
\cT\ar[r]_{f}&\cX
}
\]
commutes, combined with the definition of $r_{\Phi \cX}$ as $\Phi r_{\cX}$.)
By induction, the map $\bL\Phi^{n+1}\cX\to \bL\Phi^{n}\cX$ induces a weak equivalence 
\[
\bR F_{\cyc}(\cT,\bL\Phi^{n+1}\cX)\to \bR F_{\cyc}(\cT,\bL\Phi^{n}\cX)
\]
for all $n\geq 1$.  Since $\Phi^{\infty} \cX$ is the
homotopy limit in precyclotomic spectra, the map displayed in the
statement is also a weak equivalence.
\end{proof}

By the usual corepresentability result for $TC_{\cyc}$~\cite[6.8]{BM-cycl},
the previous proposition implies in particular that $\Phi^{\infty}
\cX\to \cX$ induces a weak equivalence on $p$-completed $TC_{\cyc}$.  In fact,
the sharper corepresentability result~\cite[6.7]{BM-cycl} implies the
following sharper result.  (We remind the reader that $TC_{\cyc}$ here means
the composite of the functor denoted $TC(-;p)$ in~\cite{BM-cycl} with
fibrant approximation.)

\begin{prop}
The natural map $\Phi^{\infty} \cX\to \cX$ in the homotopy category of
precyclotomic spectra induces a weak equivalence on $TC_{\cyc}$.
\end{prop}

\section{Descent for $TC^{\uR}$}\label{sec:descent}

In this section we study descent for $TC^{\cR}$ for a precyclotomic
base $\cR$.  If we assume that the underlying commutative ring 
precyclotomic spectrum of $\cR$ is cofibrant, then any smash power
$\cR^{(n)}$ of $\cR$ obtains the canonical structure of a
precyclotomic base. Moreover, the maps $\cR^{(m)}\to
\cR^{(n)}$ induced by (iterated) inclusions of the unit 
and (iterated) multiplication in the various factors are maps of
precyclotomic bases.  Then the usual ``Adams resolution'' cosimplicial
spectrum $\cR^{(\bullet+1)}$ is a cosimplicial
object in the category of precyclotomic bases and we can combine it with
$TC$ in various ways. 

We have in mind the case when $\cR=H\uFp$ (of Example~\ref{ex:Fp}) or
$\cR=\cMUC$ (of Example~\ref{ex:MU}) and the two theorems stated below apply
in particular to these cases.  In the
first theorem we need to assume that $\uR$ is $\aF_{p}$-connective,
which means that for every $n$, the fixed point spectrum $\uR^{C_{p^{n}}}$ is
connective (as a non-equivariant spectrum); by standard arguments,
this is equivalent to the assumption that $\Phi^{n}\uR$ is connective
(as a non-equivariant spectrum) for every $n$. 
In the second theorem, in addition to $\aF_{p}$-connectivity
of $\uR$, we also need a finite type hypothesis, and a condition on
$\pi_{0}$ called ``solid'' by 
Bousfield-Kan~\cite[I.4.5]{BousfieldKan} but (in light of emerging
terminology in the field) we call ``core'': we say that $\pi_{0}R$ is
\term{core} when the multiplication map
\[
\pi_{0}R\otimes_{\bZ}\pi_{0}R\to \pi_{0}R
\]
is an isomorphism.  More generally, we say that $\pi_{0}R$ is
\term{$p$-adically core} 
when this map becomes an isomorphism after applying the zeroth
left derived functor of $p$-completion.  We need $\pi_{0}R$ to be
$p$-adically core and also
$\pi_{0}(\Phi^{n}\uR)$ to be $p$-adically core for all $n$.

In the first descent theorem, we look at a commutative $R$-algebra
$A$. The multiplication $R^{(n)}\to R$
makes $A$ a commutative $R^{(n)}$-algebra for all $n$.  Moreover, the
maps of precyclotomic bases $\cR^{(m)}\to \cR^{(n)}$ induce maps on $TC$, 
\[
TC^{\cR^{(m)}}(A)\to TC^{\cR^{(n)}}(A).
\] 
We then have a cosimplicial object 
\[
T\supdot=TC^{\cR^{(\bullet+1)}}(A)
\]
with cofaces induced by inclusions of units and codegeneracies induced
by multiplication in the usual way (see the next section for this
functoriality of $TC$).  We prove the following theorem
below.

\begin{thm}\label{thm:descent}
Let $\cR$ be an $\aF_{p}$-connective precyclotomic base, whose
underlying commutative ring precyclotomic spectrum is cofibrant. Let
$A$ be a connective cofibrant commutative $R$-algebra.
The canonical map
\[
TC(A)\to \Tot( TC^{\cR^{(\bullet+1)}}(A))
\]
is a weak equivalence.
\end{thm}

For the second descent theorem,
we start with an arbitrary associative ring orthogonal spectrum
$A$. We can then form $THH$ relative to $\cR^{(n)}$ of the
$R^{(n)}$-algebra $R^{(n)}\sma A$:
\[
THH^{\cR^{(n)}}(R^{(n)}\sma A)
\]
and the $TC$ relative to $\cR^{(n)}$:
\[
TC^{\cR^{(n)}}(R^{(n)}\sma A).
\]
The iterated unit and multiplication maps 
$\cR^{(m)}\to \cR^{(n)}$ induce maps 
\[
TC^{\cR^{(m)}}(R^{(m)}\sma A)\to TC^{\cR^{(n)}}(R^{(n)}\sma A)
\]
and in particular, we have the augmented cosimplicial object 
\[
T\supdot=TC^{\cR^{(\bullet+1)}}(R^{(\bullet+1)}\sma A).
\]
In different notation, this construction generalizes to the case when
$\cR$ is just a commutative ring precyclotomic spectrum and not
necessarily a precyclotomic base: we have an isomorphism of
cosimplicial objects
\[
TC^{\cR^{(\bullet+1)}}(R^{(\bullet+1)}\sma A)
:=TC_{\cyc}(THH^{\cR^{(\bullet+1)}}(R^{(\bullet+1)}\sma A))
\iso
TC_{\cyc}(\cR^{(\bullet+1)}\sma THH(A)).
\]
We prove the following theorem.

\begin{thm}\label{thm:descent2}
Let $\cR$ be an $\aF_{p}$-connective cofibrant commutative ring
precyclotomic spectrum such that $p$ is not a unit in $\pi_{0}R$.  We
assume that
\begin{enumerate}
\item $\pi_{0}(\Phi^{n}\uR)$ is $p$-adically core for
all $n\geq 0$ and 
\item the underlying non-equivariant spectrum of $\Phi^{n}\uR$ is finite
$p$-type for all $n$.
\end{enumerate}
Let
$A$ be a connective associative ring orthogonal spectrum that is
cofibrant as an associative ring orthogonal spectrum or as a
commutative ring orthogonal spectrum.
Then the canonical map
\[
TC(A)\to \Tot( TC_{\cyc}(\cR^{(\bullet+1)}\sma THH(A)))
\]
is a $p$-equivalence.
If $\cR$ is an $\aF_{p}$-connective precyclotomic base, then the
canonical map
\[
TC(A)\to \Tot( TC^{\cR^{(\bullet+1)}}(R^{(\bullet+1)}\sma A))
\]
is a $p$-equivalence.
\end{thm}

\begin{proof}[Proof of Theorems~\ref{thm:descent}
and~\ref{thm:descent2}]

We note that the $TC_{\cyc}$ construction is a homotopy limit of $C_{p^{n}}$-fixed
points and $C_{p^{n}}$-fixed points commute with $\Tot$; thus, it
suffices to show that the maps
\begin{gather*}
THH(A)\to \Tot(THH^{\cR^{(\bullet+1)}}(A))\\
THH(A)\to \Tot(\cR^{(\bullet+1)}\sma THH(A))
\end{gather*}
induce a weak equivalence or $p$-equivalence on $C_{p^{n}}$
fixed point spectra for all $n\geq 0$ in the respective cases for
Theorems~\ref{thm:descent} and~\ref{thm:descent2}.  We argue by
induction on $n$.  To write a uniform argument, we let
$T\supdot$ denote either of the cosimplicial objects and prove
``$\pi$-local equivalences'' where $\pi=\bZ$ in the case of 
Theorem~\ref{thm:descent} and $\pi=\bZ/p$ in the case of
Theorem~\ref{thm:descent2}, so that a $\pi$-local equivalence means
weak equivalence in the former case and $p$-equivalence in the latter
case.

For Theorem~\ref{thm:descent2}, the base case $n=0$ follows from
standard results on convergence of the Adams spectral sequence.
For Theorem~\ref{thm:descent}, it is useful to note that 
\begin{multline*}
THH^{\cR^{(q+1)}}(A)\iso THH^{\cR}(A)\sma_{THH(A)}\dotsb\sma_{THH(A)}THH^{\cR}(A)
\\=THH^{\cR}(A)^{\sma_{THH(A)}(q+1)}
\end{multline*}
and the cosimplicial object is the Adams resolution in the category of $THH(A)$-modules 
of $THH(A)$ by the
$THH(A)$-algebra $THH^{\cR}(A)$. The hypothesis that $R$ and $A$ are connective
(and commutative) implies that the map $THH(A)\to THH^{\cR}(A)$ is a 
$1$-equivalence and from here the base case $n=0$ follows from standard
arguments: the normalized $E_{1}$-term for the homotopy group spectral
sequence is $2q$-connected in cosimplicial degree $q$, which
implies that the $q$th fiber of the cosimplicial filtration on $\Tot$
is $q$-connected.  This implies that smashing over $THH(A)$ with
$THH^{\cR}(A)$ commutes with $\Tot$, and the map of cosimplicial
objects 
\[
THH^{\cR}(A)\iso THH^{\cR}(A)\sma_{THH(A)}THH(A)\to THH^{\cR}(A)^{\sma_{THH(A)}(1+\bullet+1)}
\]
(for the constant cosimplicial object $THH^{\cR}(A)$)
is a cosimplicial homotopy equivalence.

The base case shows that the map is a non-equivariant $\pi$-local
equivalence and implies that the induced map of homotopy orbits
\begin{equation}\label{eq:orbits}
THH(A)_{hC_{p^{n}}}
\overto{\simeq_{\pi}} \Tot(T\supdot)_{hC_{p^{n}}}
\end{equation}
is a (non-equivariant) $\pi$-local equivalence in the respective
cases.  Moreover, connectivity of the fibers in the cosimplicial
filtration implies that homotopy orbits commute with $\Tot$,
\[
\Tot\bigl(T\supdot\bigr)_{hC_{p^{n}}}\overto{\simeq} 
\Tot\bigl(T\supdot_{hC_{p^{n}}}\bigr).
\]
Consider the norm cofiber sequence
\[
(T\supdot)_{hC_{p^{n}}}\to 
(T\supdot)^{C_{p^{n}}}\to 
(\Phi T\supdot)^{C_{p^{n-1}}}\to \Sigma \dotsb.
\]
Commuting fixed points and homotopy orbits with $\Tot$, up
to weak equivalence, we get a cofiber sequence
\[
\Tot(T\supdot)_{hC_{p^{n}}}\to 
\Tot(T\supdot)^{C_{p^{n}}}\to 
\Tot(\Phi(T\supdot))^{C_{p^{n-1}}}\to \Sigma \dotsb
\]
and compatible maps from the cofiber sequence
\[
THH(A)_{hC_{p^{n}}}\to THH(A)^{C_{p^{n}}}\to (\Phi
THH(A))^{C_{p^{n-1}}}\to \Sigma \dotsb.
\]
By~\eqref{eq:orbits} the map in the first position is a $\pi$-local
equivalence, and by induction the map in the third position is a
$\pi$-local equivalence: we can identify it up to weak equivalence as
the induced map on $C_{p^{n-1}}$ fixed points of the corresponding
descent problem with $\cR$ replaced by $\Phi \cR$.  It follows that
the map in the middle position is a 
$\pi$-local equivalence.  This completes the induction.
\end{proof}

\section{Functoriality of relative $TC$ in the precyclotomic base}\label{sec:TCfunc}

This section is devoted to studying the functoriality of $TC^{\cR}(A)$
in $\cR$ as well as~$A$.

Let $\cR$ be a precyclotomic base.  Then, as observed in the proof of
Theorem~\ref{thm:precycstructure}, the precyclotomic base structure
gives a homotopy that makes the diagram
\[
\xymatrix{%
\Phi THH(R)\ar[r]\ar[d]_{r_{THH(R)}}&\Phi \uR\ar[d]^{r_{\cR}}\\
THH(R)\ar[r]&\uR
}
\]
commute in the category of $\bT$-equivariant commutative ring
orthogonal spectra.  We write 
\begin{align*}
f_{0}&\colon \Phi THH(R)\to \Phi \uR\to \uR\\
f_{1}&\colon \Phi THH(R)\to THH(R)\to \uR\\
F&\colon \Phi THH(R)\otimes I\to \uR
\end{align*}
for the right-then-down composite, the down-then-right composite, and
the  homotopy, respectively.  We write $f_{i}^{*}\uR$ and $F^{*}\uR$ for
the $\bT$-equivariant commutative $\Phi THH(R)$- and $(\Phi THH(R)\otimes
I)$-algebra structures on $\uR$ induced by $f_{i}$ and $F$.  We use the
notation 
\[
c\colon THH(R)\to \uR
\]
for the canonical map of Definition~\ref{defn:eqrelthh}.

For the functoriality used in Theorem~\ref{thm:descent}, even when $A$
is cofibrant as an $R$-algebra, we cannot expect $A$ to be cofibrant
as an $R^{(n)}$-algebra for all $n$, and so we need to use a point-set
model for the derived functor of $THH^{\uR}(A)$ that requires minimal
cofibrancy on $A$.  

First we need to ensure that $THH(A)$ has the right $\bT$-equivariant
($\aF_{\fin}$-colocal) homotopy type.  For this, it suffices to assume
that $A$ is flat for the smash product in orthogonal spectra (meaning
that $A\sma (-)$ preserves all weak equivalences of orthogonal
spectra); moreover, when $A$ is flat for the smash product in
orthogonal spectra, $THH(A)$ is flat for the smash product in
$\bT$-equivariant orthogonal spectra.  To get the cyclotomic
structure map on $THH(A)$, the construction of~\cite[4.2]{ABGHLM}
requires an additional hypothesis; for this, it suffices that the
underlying orthogonal spectrum of $A$ be
in the (non-equivariant) class $\nEQC$ discussed in
Section~\ref{sec:equivariant}.   To ensure both that $A$ is flat for
the smash product in orthogonal spectra and that the underlying
orthogonal spectrum of $A$ is in the class $\nEQC$, it is enough that $A$ be a
cofibrant $R$-algebra for some commutative ring orthogonal
spectrum $R$ that is cofibrant in the standard model structure or 
the model structure of~\cite[III.4.1]{BM-cycl23}.  

Next we need a point-set model for the derived smash product
$THH(A)\sma_{THH(R)}\uR$.  Let $\uM$ be a $\bT$-equivariant orthogonal
spectrum and $\uR$ be a cofibrant $\bT$-equivariant commutative ring
orthogonal spectrum in the model structure of~\cite[Thm.~C]{BM-cycl23}.  If
$\uR'$ is a cofibrant commutative $\uR$-algebra, then $\uM\sma_{\uR}\uR'$ represents the derived smash product.  The
structure of a commutative $\uR$-algebra $\uR'$ is just a map of
$\bT$-equivariant commutative ring orthogonal spectra $g\colon \uR\to \uR'$
and $\uR'$ is a cofibrant commutative $\uR$-algebra exactly when this
map is a cofibration.  In the case when $\uR'$ is not necessarily
cofibrant as a commutative $\uR$-algebra but is cofibrant as a
$\bT$-equivariant commutative ring orthogonal spectrum, we can use the
mapping cylinder construction to produce a homotopy equivalent
cofibrant commutative $\uR$-algebra.  Let 
\[
Ig=(\uR\otimes I)\cup_{(\uR\otimes\{1\})}\uR'
\]
with the pushout done in the category of $\bT$-equivariant commutative
ring orthogonal spectra (i.e., ``$\cup$''$=\sma$), gluing along the
given map $g$:
\[
\uR\otimes\{1\}\iso \uR\overto{g}\uR'.
\]
Returning to $THH^{\uR}(A)$, when $\uR$ is a cofibrant
$\bT$-equivariant commutative ring orthogonal spectrum, $Ic\to \uR$ is
a cofibrant approximation in the category of commutative
$THH(R)$-algebras, and when in addition the underlying orthogonal
spectrum of the $R$-algebra $A$ is flat
for the smash product in orthogonal spectra,
$THH(A)\sma_{THH(R)}Ic$ then represents the derived functor
$THH^{\uR}(A)$. 

Assume the underlying orthogonal spectra of $R$ and $A$ are in the
(non-equivariant) class $\nEQC$ of Section~\ref{sec:equivariant}
so that $THH(R)$ and $THH(A)$ have cyclotomic structures and consider
the following zigzag of maps. 
\[ \small
\xymatrix@C-5pc{%
\Phi (THH(A)\sma_{THH(R)}Ic)\ar@{<-}[d]
&&\relax\hspace{-4em}THH(A)\sma_{THH(R)}Ic
\\
\Phi THH(A)\sma_{\Phi THH(R)}I(\Phi c)\ar@{.>}[d]
&\hspace{6em}\Phi THH(A)\sma_{\Phi THH(R)\otimes I}IF
\hspace{6em}
\\
\Phi THH(A)\sma_{\Phi THH(R)}If_{0}\ar[ur]
&&\hspace{-4em}
\Phi THH(A)\sma_{\Phi THH(R)}If_{1}
  \ar@<1em>@{-->}[uu]\ar[ul]
}
\]
When $\uR$ is cofibrant as a $\bT$-equivariant commutative ring
orthogonal spectrum in the model
structure of~\cite[Thm.~C]{BM-cycl23}, $THH(R)$, $\Phi THH(R)$, and $\Phi \uR$
are also cofibrant. By the discussion above, when in addition $A$ is
flat for the smash product in orthogonal spectra, 
then the solid arrows are weak equivalences and the dashed arrow is an
$\aF_{\fin}$-equivalence.  When in addition $\cR$
is cyclotomic, the dotted arrow is an $\aF_{p}$-equivalence.  This
zigzag motivates the following definition.

\begin{defn}
We define the category of zigzag-cyclotomic spectra as follows.  An
object consists of $\bT$-equivariant orthogonal spectra
$\uX,\uX_{1},\uX_{2},\uX_{3}$ and $\bT$-equivariant maps
\[
\Phi \uX\overfrom{r_{3}} \uX_{3}\overto{r_{2}} \uX_{2}\overfrom{r_{1}} \uX_{1}\overto{r_{0}} \uX.
\]
A map of zigzag-cyclotomic spectra $\cf\colon \cX\to \cY$ consists of maps of
$\bT$-equivariant orthogonal spectra $f,f_{1},f_{2},f_{3}$ as pictured,
making the following diagram commute.
\[
\xymatrix{%
\Phi \uX\ar@{<-}[r]^{r_{3}}\ar[d]_{\Phi f}
&\uX_{3}\ar[r]^{r_{2}}\ar[d]^{f_{3}}
&\uX_{2}\ar@{<-}[r]^{r_{1}}\ar[d]^{f_{2}}
&\uX_{1}\ar[r]^{r_{0}}\ar[d]^{f_{1}}
&\uX\ar[d]^{f}\\
\Phi \uY\ar@{<-}[r]_{r_{3}}
&\uY_{3}\ar[r]_{r_{2}}
&\uY_{2}\ar@{<-}[r]_{r_{1}}
&\uY_{1}\ar[r]_{r_{0}}&\uY
}
\]
For $\aF$ any family of subgroups of $\bT$, an $\aF$-equivalence is a
map where each of $f,f_{1},f_{2},f_{3}$ is an $\aF$-equivalence.
\end{defn}

In order to do a $TC$ construction, we need an $\Omega$-spectrum
replacement functor in this category. For this it suffices to have an
$\Omega$-spectrum replacement functor $R$ (or more precisely, functor
$R$ and natural transformation $\eta\colon \Id\to R$) in the category of
$\bT$-equivariant orthogonal spectra that comes with a natural
transformation $\theta \colon \Phi R\to R\Phi$ that makes the following diagram
commute.
\[
\xymatrix@C-1pc{%
&\Phi \uX\ar[dl]_{\Phi \eta_{\uX}}\ar[dr]^{\eta_{\Phi \uX}}\\
\Phi R\uX\ar[rr]_{\theta_{\uX}}&&R\Phi \uX
}
\]
We call such a structure a $\Phi$-compatible $\Omega$-spectrum
replacement functor.  Theorem~4.7 of~\cite{BM-tc} and its proof assert
the existence of such functors (and construct two).  We now choose and
fix a $\Phi$-compatible $\Omega$-spectrum replacement functor $R$.

We construct a functor $TC_{z}$ from zigzag-cyclotomic spectra to
orthogonal spectra as follows.

\begin{cons}
For a zigzag-cyclotomic spectrum $\cX$, define $TR_{z}(\cX)$ to be the
homotopy limit (constructed via the Bousfield-Kan cobar construction
or mapping microscope) of the following diagram.
\[\small
\xymatrix@+1pc{%
&(R\uX_{3})^{C_{p^{n}}}\ar[dr]|-{(Rr_{2})^{C_{p^{n}}}}
&(R\uX_{1})^{C_{p^{n}}}\ar[dr]|-{\ \ (Rr_{0})^{C_{p^{n}}}}&&\cdots\\
\relax\cdots\ar[r]
&(R\Phi \uX)^{C_{p^{n}}}\ar@{<-}[u]|-{(Rr_{3})^{C_{p^{n}}}}
&(R\uX_{2})^{C_{p^{n}}}\ar@{<-}[u]|-{(Rr_{1})^{C_{p^{n}}}}
&(R\uX)^{C_{p^{n}}}\ar[r]
&(R\Phi \uX)^{C_{p^{n-1}}}\ar@{<-}[u]|-{(Rr_{3})^{C_{p^{n-1}}}}
}
\]
where the unlabeled map $(R\uX)^{C_{p^{n}}}\to (R\Phi \uX)^{C_{p^{n-1}}}$
(for each $n\geq 1$) is the composite
\[
(R\uX)^{C_{p^{n}}}\iso 
(\rho^{*} (R\uX)^{C_{p}})^{C_{p^{n-1}}}\to
(\Phi R\uX)^{C_{p^{n-1}}}\to
(R\Phi \uX)^{C_{p^{n-1}}}
\]
induced by the canonical map from the categorical fixed points to the geometric
fixed points and the $\Phi$-compatibility structure of $R$.
Naturality of the inclusion of fixed points map $\mathrm{F}$ implies
that the diagrams
\[ \small
\xymatrix{%
(R\Phi \uX)^{C_{p^{n}}}\ar@{<-}[r]^(.55){(Rr_{3})^{C_{p^{n}}}}
\ar[d]_-{\mathrm F}
&(R\uX_{3})^{C_{p^{n}}}\ar[r]^(.55){(Rr_{2})^{C_{p^{n}}}}
\ar[d]_-{\mathrm F}
&(R\uX_{2})^{C_{p^{n}}}\ar@{<-}[r]^(.55){(Rr_{1})^{C_{p^{n}}}}
\ar[d]_-{\mathrm F}
&(R\uX_{1})^{C_{p^{n}}}\ar[r]^(.55){(Rr_{0})^{C_{p^{n}}}}
\ar[d]_-{\mathrm F}
&(R\uX)^{C_{p^{n}}}
\ar[d]_-{\mathrm F}
\\
(R\Phi \uX)^{C_{p^{n-1}}}\ar@{<-}[r]_(.45){(Rr_{3})^{C_{p^{n-1}}}}
&(R\uX_{3})^{C_{p^{n-1}}}\ar[r]_(.45){(Rr_{2})^{C_{p^{n-1}}}}
&(R\uX_{2})^{C_{p^{n-1}}}\ar@{<-}[r]_(.45){(Rr_{1})^{C_{p^{n-1}}}}
&(R\uX_{1})^{C_{p^{n-1}}}\ar[r]_(.45){(Rr_{0})^{C_{p^{n-1}}}}
&(R\uX)^{C_{p^{n-1}}}
}
\]
and
\[ \small
\xymatrix{%
(R\uX)^{C_{p^{n+1}}}\ar[r]
\ar[d]_-{\mathrm F}
&(R\Phi \uX)^{C_{p^{n}}}
\ar[d]_-{\mathrm F}
\\
(R\uX)^{C_{p^{n}}}\ar[r]
&(R\Phi \uX)^{C_{p^{n-1}}}
}
\]
commute and so $\mathrm F$ induces a self-map $\mathrm F$ of
$TR_{z}(\cX)$.  Define $TC_{z}(\cX)$ to be the homotopy equalizer of
$\mathrm{F}$ and the identity on $TR_{z}(\cX)$.
\end{cons}

The following is clear from construction.

\begin{prop}\label{prop:TCzisderived}
If $\cX\to \cY$ is an $\aF_{p}$-equivalence of zigzag-cyclotomic
spectra, then the induced maps on $TR_{z}$ and $TC_{z}$ are weak
equivalences. 
\end{prop}

To compare $TC_{\cyc}$ and $TC_{z}$, we use the following functor from
precyclotomic spectra to zigzag-cyclotomic spectra.

\begin{defn}
Given a precyclotomic spectrum $\cX$, let $z\cX$ be the zigzag-cyclotomic spectrum
\[
\Phi \uX\overfrom{=}\Phi \uX\overto{r_{\cX}} \uX\overfrom{=}\uX\overto{=}\uX.
\]
\end{defn}

\begin{thm}\label{thm:compTCz}
There is a zigzag of natural weak equivalences connecting
$TC_{\cyc}(\cX)$ and $TC_{z}(z\cX)$.
\end{thm}

\begin{proof}
Let $TR_{\cyc}(\cX)$ be the homotopy limit of $(R\uX)^{C_{p^{n}}}$
under the $\mathrm R$ maps
\begin{multline*}
\mathrm{R}\colon 
(R\uX)^{C_{p^{n}}}\iso (\rho^{*}(R\uX)^{C_{p}})^{C_{p^{n-1}}}
\to (\Phi R\uX)^{C_{p^{n-1}}}\\\overto{\theta^{C_{p^{n-1}}}}
(R\Phi\uX)^{C_{p^{n-1}}}
\overto{Rr_{\cX}^{C_{p^{n-1}}}}(R\uX)^{C_{p^{n-1}}}.
\end{multline*}
Then $TC_{\cyc}(\cX)$ is naturally weakly equivalent to the homotopy
equalizer of the self-maps induced by $\mathrm{F}$ and $\mathrm{R}$ on
$TR_{\cyc}(\cX)$, and since the self-map of $TR_{\cyc}(\cX)$ induced
by $\mathrm{R}$ is naturally homotopic to the identity, $TC_{\cyc}$ is
naturally weakly equivalent to the homotopy equalizer of the identity
and the self-map induced by $\mathrm{F}$ on $TR_{\cyc}(\cX)$.  We
defined $TC_{z}(z\cX)$ as the homotopy equalizer of the identity and
the self-map induced by $\mathrm{F}$ on $TR_{z}(z\cX)$, and so it
suffices to construct a natural weak equivalence $TR_{\cyc}(\cX)\to
TR_{z}(z\cX)$ that is compatible up to natural homotopy with
$\mathrm{F}$.  

For the diagram defining $TR_{z}$, consider the subdiagram that
consists of $(R\uX)^{C_{p^{n}}}$ and everything to the right of it as
displayed above
\[ \small
\xymatrix@-1pc{%
(R\uX_{1})^{C_{p^{n}}}\ar[dr]
&&(R\uX_{3})^{C_{p^{n-1}}}\ar[d]\ar[dr]
&(R\uX_{1})^{C_{p^{n-1}}}\ar[d]\ar[dr]
\\
&(R\uX)^{C_{p^{n}}}\ar[r]
&(R\Phi \uX)^{C_{p^{n-1}}}
&(R\uX_{2})^{C_{p^{n-1}}}
&\cdots
}
\]
The homotopy limit of this diagram can be constructed as an
iterated homotopy pullback and the point-set limit is an iterated
point-set pullback.  
In the particular case of $z\cX$, all the vertical non-diagonal maps
are the identity, and so the point-set limit maps to
$(R\uX)^{C_{p^{n}}}$ by an isomorphism, and this maps to the homotopy
limit by a homotopy equivalence.  Under these isomorphisms (for
$n$ varying), the maps induced on the limits by inclusion of
subdiagrams are the iterated maps
\[
\mathrm{R}^{n-m}\colon (R\uX)^{C_{p^{n}}}\to (R\uX)^{C_{p^{m}}}.
\]
The map from the sequential homotopy limit of point-set inverse limits
to the homotopy inverse limit of the whole diagram
then induces a weak equivalence from $TR_{\cyc}(\cX)$ to $TR_{z}(z\cX)$ that is
compatible with the $\mathrm{F}$ self-maps (without needing a
homotopy). 
\end{proof}

We now return to the problem of constructing $TC^{\cR}(A)$ as a point-set
functor.  We require a functor of $\cR$ and $A$ in the following sense.

\begin{defn}\label{defn:basepair}
Let $\aBPair$ denote the category where
\begin{itemize}
\item the objects are ordered pairs $(\cR,A)$ with $\cR$ a
precyclotomic base and $A$ an associative $R$-algebra, where $\uR$ is
cofibrant in the category of $\bT$-equivariant commutative ring
orthogonal spectra (for the standard model structure or the model
structure of~\cite[Thm.~C]{BM-cycl23}) and the underlying orthogonal spectrum
of $A$ is in the (non-equivariant) class $\nEQC$ of
Section~\ref{sec:equivariant} and is flat for the smash product in
orthogonal spectra.
\item the set of maps in $\aBPair$ from
$(\cR,A)$ to $(\cR',A')$ consists of the set of ordered pairs of maps
$\cf\colon \cR\to \cR'$ and $g\colon A\to f^{*}A'$ where $\cf$ is a
map of precyclotomic bases 
and $g$ is a map of $R$-algebras for $f^{*}A'$ the $R$-algebra
given by $A'$ with the $R$-algebra structure induced by the map
$f\colon R\to R'$ underlying $\cf$.
\end{itemize}
We say that a map $(\cf,g)$ in $\aBPair$ is a weak equivalence when
the map $\cf$ is a weak equivalence of precyclotomic spectra (i.e.,
$\uf$ is an $\aF_{p}$-equivalence of the underlying $\bT$-equivariant
orthogonal spectra) and $g$ is a weak equivalence.
\end{defn}

As motivated by the discussion above, we have the following functor
from $\aBPair$ to zigzag-cyclotomic spectra.

\begin{cons}\label{cons:zigzag}
Let $(\cR,A)$ be an object in $\aBPair$.  Define the zigzag-cyclotomic
object $THH_{z}^{\cR}(A)$ by
\begin{align*}
\uX_{3}&:=\Phi THH(A)\sma_{\Phi THH(R)}I(\Phi c)\\
\uX_{2}&:=\Phi THH(A)\sma_{\Phi THH(R)\otimes I}IF\\
\uX_{1}&:=\Phi THH(A)\sma_{\Phi THH(R)}If_{1}\\
\uX&:=THH(A)\sma_{THH(R)}Ic
\end{align*}
where the maps are as described at the start of the section.  We note
that this is functorial for $(\cR,A)$ in $\aBPair$ and sends weak
equivalences to $\aF_{p}$-equivalences.
\end{cons}

If we fix a precyclotomic base $\cR$ whose underlying $\bT$-equivariant
commutative ring orthogonal spectrum is cofibrant (in either the
standard model structure or the model structure of~\cite[Thm.~C]{BM-cycl23}),
then we get a functor from cofibrant $R$-algebras to $\aBPair$
(sending $A$ to $(\cR,A)$).  The following theorem asserts that
$TC_{z}(THH_{z}^{\cR}(A))$ gives a point-set 
model of $TC^{\cR}(A)$.  As a consequence, $TC_{z}(THH_{z}^{\cR}(A))$ extends  
$TC^{\cR}(A)$ to be functorial in both $\cR$ and $A$.

\begin{thm}\label{thm:zigzag}
Let $\cR$ be a precyclotomic base whose underlying $\bT$-equivariant
commutative ring orthogonal spectrum is cofibrant in the
model structure of~\cite[Thm.~C]{BM-cycl23}.  Then 
the composite functor $TC_{z}(THH_{z}^{\cR}(-))$ viewed as a functor
from cofibrant $R$-algebras to the stable category is naturally isomorphic to the
functor $TC^{\cR}(-)$.
\end{thm}

The proof of the theorem fills the remainder of this section. 

As per the statement $\cR$ is fixed, and as in
Section~\ref{sec:relcyc} we fix a cofibrant approximation $q\colon
\cT\to THH(R)$ in the category of commutative ring precyclotomic
spectra, a map of commutative ring precyclotomic spectra $g\colon
\cT\to \cR$ and a path $H$ from the image of $g$ in the homotopy
equalizer
\[
\xymatrix@C-1pc{%
r_{\cT}^{*},r_{\cR*}\circ \Phi\colon \aCom^{\bT}(\uT,\uR)\ar@<-.5ex>[r]\ar@<.5ex>[r]
&\aCom^{\bT}(\Phi \uT,\uR)
}
\]
to the element specified by $c\circ q$ and $F\circ q$.  Such a path is
adjoint to a map of $\bT$-equivariant commutative ring orthogonal spectra
\[
\Phi \uT\otimes I^{2}\to \uR
\]
which by abuse of notation, we will also denote by $H$.  The four
faces of the square $I^{2}$ are the following homotopies $\Phi
\uT\otimes I\to \uR$:
\begin{align*}
[0,1]\times \{0\}&: F\circ q\\
[0,1]\times \{1\}&: K=\text{constant homotopy }g\circ r_{\cT}=r_{\cR}\circ \Phi g\\
\{0\}\times [0,1]&: r_{\cR}\circ \Phi G\\
\{1\}\times [0,1]&: G\circ r_{\cT}
\end{align*}
where $G$ is the homotopy from $g$ to $c\circ q$, as in the notation
in the proof of Theorem~\ref{thm:precycstructure}.  Since $\uR$ is
cofibrant as a $\bT$-equivariant commutative ring orthogonal spectrum,
the derived functor $THH(A)\sma_{\uT}^{\bL}g^{*}\uR$ in that proof is
represented by $THH(A)\sma_{\uT}I(g)$ (when $A$ is a cofibrant 
$R$-algebra as in the hypothesis of Theorem~\ref{thm:zigzag}), where
the action of $\uT$ on $THH(A)$ is always via $q\colon \uT\to THH(R)$.

To prove Theorem~\ref{thm:zigzag}, it suffices to construct a
natural zigzag of $\aF_{p}$-equivalences of zigzag-cyclotomic spectra
relating $\cX=THH_{z}^{\cR}(-)$ and $\cZ=z(THH(-)\sma_{\uT}I(g))$.
Let $\cY$ be the functor from cofibrant $R$-algebras to zigzag-cyclotomic spectra defined by
\begin{align*}
\uY_{3}&=\Phi THH(-)\sma_{\Phi \uT}I(\Phi g)\\
\uY_{2}&=\Phi THH(-)\sma_{\Phi \uT\otimes I}IK\\
\uY_{1}&=\Phi THH(-)\sma_{\Phi \uT}I(r_{\cR}\circ \Phi g)\\
\uY&=THH(-)\sma_{\uT}Ig
\end{align*}
with maps
\begin{itemize}
\item $\Phi \uY \from \uY_{3}$ induced by the lax symmetric monoidal structure map for $\Phi$;
\item $\uY_{3} \to \uY_{2}$ the composite
\begin{multline*}
\qquad \qquad  \Phi THH(-)\sma_{\Phi \uT}I(\Phi g)
\to \Phi THH(-)\sma_{\Phi \uT}I(r_{\cR}\circ \Phi g)\\
\to \Phi THH(-)\sma_{\Phi \uT\otimes I}IK
\end{multline*}
where the second map is induced by the inclusion of $\uT\otimes \{0\}$
in $\uT\otimes I$ and the first map is induced by the map $I(\Phi
g)\to I(r_{\cR}\circ \Phi g)$ (induced by $r_{\cR}\colon \Phi
\uR\to \uR$);
\item $\uY_{2}\from \uY_{1}$ induced by the inclusion of $\uT\otimes
\{1\}$ in $\uT\otimes I$; and
\item $\uY_{1}\to \uY$ induced by the precyclotomic structure maps 
$r_{THH(-)}$ and $r_{\cT}$ (using the precyclotomic equation $r_{\cR}\circ \Phi
g=g\circ r_{\cT}$). 
\end{itemize}
We then have a natural map of zigzag-cyclotomic spectra 
\[
\cY\to \cZ=z(THH(-)\sma_{\uT}I(g))
\]
with component maps 
\[
\uY_{3}=\Phi THH(-)\sma_{\Phi \uT}I(\Phi g) \to \Phi (THH(-)\sma_{\uT}Ig)=\uZ_{3}
\]
induced by the lax symmetric monoidal structure map of $\Phi$;
\[
\uY_{2}=\Phi THH(-)\sma_{\Phi \uT\otimes I}IK \to THH(-)\sma_{\uT}Ig=\uZ_{2}
\]
induced by $r_{THH(-)}$, $r_{\cT}$, and the collapse map $\uT\otimes
I\to \uT$;
\[
\uY_{1}=\Phi THH(-)\sma_{\Phi \uT}I(r_{\cR}\circ \Phi g)\to 
THH(-)\sma_{\uT}Ig=\uZ_{1}
\]
the map induced by $r_{THH(-)}$ and $r_{\cT}$; and 
\[
\uY=THH(-)\sma_{\uT}Ig=\uZ
\]
the identity.  Because $THH(A)$ and $\cT$ are cyclotomic,
each of these maps is an $\aF_{p}$-equivalence.

We construct the natural map of zigzag-cyclotomic spectra $\cY\to
\cX=THH_{z}^{\cR}(-)$ as follows.  The map 
\[
\uY_{3}=\Phi THH(-)\sma_{\Phi \uT}I(\Phi g) \to \Phi THH(-)\sma_{\Phi THH(R)}I(\Phi c)=\uX_{3}
\]
is induced by the map $q\colon \uT\to THH(R)$ and a map of
commutative $\uT$-algebras $I(\Phi g)\to (\Phi q)^{*}I(\Phi c)$ we now
define.  Using the multiplication by 2 isomorphism $[0,1]\iso [0,2]$,
$I(\Phi g)$ is isomorphic to 
\[
(\Phi \uT\otimes [0,1])\cup_{\Phi \uT\otimes \{1\}}
(\Phi \uT\otimes [1,2])\cup_{\Phi \uT\otimes \{2\}}
(\Phi g)^{*}\uR
\]
(where the pushout is done in $\bT$-equivariant commutative ring orthogonal
spectra, i.e., ``$\cup$''$=\sma$).  The $\Phi \uT\otimes
[0,1]$ piece maps by $\Phi q\otimes [0,1]$, and the $\Phi \uT\otimes
[1,2]$ piece maps 
using the subtract 1 isomorphism $[1,2]\to [0,1]$ and the homotopy
$\Phi G$ (which starts at $\Phi c\circ \Phi q$ and ends at $\Phi g$).
The map 
\[
\uY_{2}=\Phi THH(-)\sma_{\Phi \uT\otimes I}IK \to 
\Phi THH(-)\sma_{\Phi THH(R)\otimes I}IF=\uX_{2}
\]
is induced by the map of commutative $(\cT\otimes I)$-algebras 
\[
IK\iso
(\Phi \uT\otimes I\otimes [0,1])\cup_{\Phi \uT\otimes \{1\}}
(\Phi \uT\otimes I\otimes [1,2])\cup_{\Phi \uT\otimes \{2\}}
K^{*}\uR
\to (\Phi q\otimes I)^{*}IF
\]
that sends $\Phi \uT\otimes I\otimes [0,1]$ by $\Phi q\otimes I\otimes
[0,1]$ and maps $\Phi \uT\otimes I\otimes [1,2]$ using $H$ (which can
be viewed as a homotopy from $F\circ q$ to $K\colon \Phi \uT\otimes
I\to \uR$).   The maps
\begin{gather*}
\uY_{1}=\Phi THH(-)\sma_{\Phi \uT}I(r_{\cR}\circ \Phi g)
\to \Phi THH(-)\sma_{\Phi THH(R)}I(f_{1})=\uX_{1}\\
\uY=THH(-)\sma_{\uT}Ig
\to THH(-)\sma_{THH(R)}Ic=\uX
\end{gather*}
are defined in the same way as the maps above but with $\uY_{1}\to
\uX_{1}$ using $\Phi q\otimes I$ and $G\circ r_{\cT}$ and $\uY\to \uX$
using $q\otimes I$ and $G$.  A straightforward check of diagrams
(using the precyclotomic equations $q\circ r_{\cT}=r_{THH(R)}\circ
\Phi q$ and $g\circ r_{\cT}=r_{\cR}\circ \Phi g$) shows that these
maps construct a map of zigzag-cyclotomic spectra.  Because the map
$\uT\to THH(R)$ is by definition an $\aF_{p}$-equivalence, the
resulting map $\cY\to \cX$ is an $\aF_{p}$-equivalence of zigzag-cyclotomic spectra.


\bibliographystyle{plain}
\bibliography{bluman}

\end{document}